\documentclass[12pt]{article}
\usepackage{amsmath,latexsym,amssymb,amsthm,enumerate,amsmath,amscd}
\usepackage{amsthm,amsmath,amsfonts,amssymb,amscd,latexsym,stmaryrd}
\usepackage[mathscr]{eucal}
\usepackage{color}
\usepackage[usenames,dvipsnames]{xcolor}
\usepackage{colortbl}
\usepackage[all,cmtip]{xy}
\usepackage{graphicx}
\usepackage{comment}
\usepackage{amstext,amsthm,amsmath,amsfonts,amssymb,amscd,latexsym,txfonts,stmaryrd}
\usepackage{color}
\usepackage{epsfig}
\usepackage{cite}
\usepackage[all,cmtip]{xy}
\usepackage{appendix}
\usepackage{hyperref}
\usepackage{tikz}

\newcommand{\Gr}{\operatorname{Gr}}

\newcommand{\Sym}{\operatorname{Sym}}
\newcommand{\C}{{\mathbb C}}
\newcommand{\R}{{\mathbb R}}
\newcommand{\Z}{{\mathbb Z}}

\newcommand{\QQ}{{\mathbb Q}}

\theoremstyle{plain}
\newtheorem{theorem}{Theorem}[section]
\newtheorem{proposition}[theorem]{Proposition}
\newtheorem{corollary}[theorem]{Corollary}
\newtheorem{lemma}[theorem]{Lemma}
\theoremstyle{definition}
\newtheorem{definition}[theorem]{Definition}
\newtheorem{conjecture}[theorem]{Conjecture}

\newtheorem{remark}[theorem]{Remark}
\newtheorem{example}[theorem]{Example}

\newtheorem{question}[theorem]{Question}

\newtheorem{thm}{Theorem}

\newtheorem{rem}[thm]{Remark}

\newtheorem*{remark*}{Remark}
\newtheorem*{uthm}{Theorem}

\newtheorem*{ack}{Acknowledgment}

\numberwithin{equation}{section}
\numberwithin{table}{section}
\newtheorem*{problem*}{Problem}
\newtheorem*{question*}{Question}
\newtheorem*{questionrephrased*}{Question 1.2 Rephrased}
\newtheorem*{example*}{Example}
\newtheorem*{watthm*}{Watanabe's Theorem}
\newtheorem*{remarks*}{Remarks}
\newtheorem*{claim*}{Claim}

\newtheorem*{proposition*}{Proposition}
\newtheorem*{lemma*}{Lemma}
\newtheorem*{conjecture*}{Conjecture}

\newtheorem*{JBC*}{Watanabe's Bold Conjecture}

\newtheorem*{UPTP*}{Universal Property of Tensor Products}
\newtheorem*{problem1*}{Problem 1}
\newtheorem*{problem2*}{Problem 2}
\newtheorem*{fact*}{Fact}

\newtheorem*{F1*}{Fact 1}
\newtheorem*{F2*}{Fact 2}
\newtheorem*{F3*}{Fact 3}

\definecolor{purple}{rgb}{0.4,0.2,0.4}

\def\Gr{\mathrm{Gr}}

\def\cha{\mathrm{char}\ }

\def\<{\left<}
\def\>{\right>}
\def\Gl{\mathrm{Gl}}
\def\A{\mathscr{A}}

\def\C{\mathbb{C}}

\def\F{\sf{k}}

\def\R{\mathbb{R}}
\def\Stab{\mathrm{Stab}}
\def\Sym{\mathrm{Sym}}

\def\ns{\footnotesize \it}

\def\cha{\mathrm{char}\ }

\begin{document}
\title{Free extensions and Lefschetz properties, with an application to rings of relative coinvariants}

\author{Chris McDaniel\\[0.05in]
	{\ns Endicott College, 376 Hale St
		Beverly, MA 01915, USA.}\\[.2in]
 Shujian Chen\footnote{Current address: Department of Mathematics, Brandeis University}\\[.05in]
 {\ns Department of Mathematics, Northeastern University, Boston, MA 02115, USA.
}\\[.2in] 
Anthony Iarrobino\\[.05in]
{\ns Department of Mathematics, Northeastern University, Boston, MA 02115,
 USA.}\\[.2in]
 Pedro Macias Marques\\[.05in]
{\ns Departamento de Matem\'{a}tica, Escola de Ci\^{e}ncias e Tecnologia, Centro de Investiga\c{c}\~{a}o}\\[-.05in]
{\ns  em Matem\'{a}tica e Aplica\c{c}\~{o}es, Instituto de Investiga\c{c}\~{a}o e Forma\c{c}\~{a}o Avan\c{c}ada,}\\[-.05in]
{\ns Universidade de \'{E}vora, Rua Rom\~{a}o Ramalho, 59, P--7000--671 \'{E}vora, Portugal.}}
\date{December 7, 2018}
\maketitle

\begin{abstract}Graded Artinian algebras can be regarded as algebraic analogues of cohomology rings (in even degrees) of compact topological manifolds.  In this analogy, a free extension of a base ring with a fiber ring corresponds to a fiber bundle over a manifold.  For rings, as with manifolds, it is a natural question to ask:  to what extent do properties of the base and the fiber carry over to the extension?  For example, if the base and fiber both satisfy a strong Lefschetz property, can we conclude the same for the extension?  Or, more generally, can one determine the generic Jordan type for the extension given the generic Jordan types of the base and fiber?    
We address these questions using the relative coinvariant rings as prototypical models.  We show that if $V$ is a vector space and if the subgroup $W$ of the general linear group $\Gl (V)$, is a non-modular finite reflection group and $K\subset W$ is a non parabolic reflection subgroup, then the relative coinvariant ring $R^K_W$ cannot have a linear element of strong Lefschetz Jordan type. However, we give examples where these rings $R^K_W$, some with non-unimodal Hilbert functions, nevertheless have (non-homogeneous) elements of strong Lefschetz Jordan type.  Some of these examples give rise to open questions concerning Lefschetz properties of certain algebras $A(m,n)$, related to combinatorial questions proposed and partially answered by G. Almqvist.\end{abstract}
\section{Introduction.}
There has been substantial work in the last decades on the Lefschetz properties of graded Artinian algebras $A$; more recently there has been study of the Jordan type of multiplication matrices -- the partition determined by the similarity class of multiplication by a non-unit element $\ell\in \mathfrak{m}_A$. The Jordan type of multiplication by a generic linear form $\ell\in A_1\subset\mathfrak{m}_A$ and the Hilbert function of a graded $A$ determine whether $A$ is strong or weak Lefschetz, or neither. Several authors have provided examples of algebras that are weak Lefschetz but not strong Lefschetz, or have studied the non-strong Lefschetz locus of certain graded Artinian algebras -- see \cite{Gondim,Guerr,GonZ,MW,Yosh,BMMN,AIK}, also \cite{IMM1,IMM3} and references cited there.  A portion of these articles have considered Lefschetz properties for graded Artinian algebras arising as coinvariant algebras of groups acting on polynomial rings, see for example \cite{McD,HWW,MNW,Yosh}. Our paper builds on those of T. Maeno, Y. Numata and A. Wachi  \cite {MNW} and of the first author \cite{McD}, who  consider Lefschetz properties for coinvariant rings.  In this paper we study Lefschetz properties and Jordan types of a related class of graded Artinian rings that we term relative coinvariant rings: the quotient of the invariant ring of a finite subgroup by the ideal of invariants a larger finite group containing it. In particular we apply the theory of Jordan types for the free extensions of T. Harima and J. Watanabe \cite{HW}, and the notion of strong Lefschetz type elements of \cite{IMM1,IMM2} to study certain relative coinvariant algebras for pairs of finite reflection groups, which are non-standard graded Artinian complete intersection algebras. We study an infinite series of relative coinvariant rings which are not strong Lefschetz, but which have (non-homogeneous) elements of strong Lefschetz type. Questions concerning the Lefschetz properties of such algebras connect to long open combinatorial problems.  \vskip 0.2cm\noindent

{\it Notation}. Let ${\F}$ be an arbitrary field; let $A$ be a $\Z_{\ge 0}$-graded Artinian algebra, with maximal ideal $\mathfrak m_A=\oplus_{i\ge 1} A_i$ and satisfying $A_0={\sf k}$ (i.e. $A$ is connected over $\sf k$). We say that $A$ has the standard grading if $\mathfrak m_A$ is generated over $\F$ by $A_1$, and otherwise it is non-standard. 
We denote by $ \A$ a local Artinian algebra with maximal ideal ${\mathfrak m_{\A}}$ and satisfying $\A/{\mathfrak m_{\A}}\cong \sf k$.  Note that a graded Artinian algebra $A$ is also a local Artinian algebra; we denote by $\kappa(A)=\A$ that (ungraded) local Artinian algebra obtained by simply forgetting the grading on $A$.\label{kappapage}   
All graded modules $M$ over $A$ are $\Z$-graded, so $M=\oplus_{i\in\Z}M_i$; we denote by $M_+=\oplus_{i> 0}M_i$.  All maps between graded objects $\phi\colon M\rightarrow N$ are degree-preserving homomorphisms.  The symbol $M(m)$ represents the graded object $M$ shifted up by $m$, i.e. $M(m)_i=M_{m+i}$.  The \emph{socle degree} or \emph{formal dimension} $j_A$ of $A$ is the largest integer $j$ such that $A_j\not=0$.  For a local Artinian algebra $\mathscr{A}$, the socle degree is the largest degree $j$ for which $\mathfrak{m}_{\A}^j\neq 0$; it is the socle degree of the \emph{associated graded algebra} 
$$\Gr_{\mathfrak m_{\mathscr{A}}}(\A)=\bigoplus_{i=0}^{j_{\A}}\mathfrak{m}_{\A}^i/\mathfrak{m}_{\A}^{i+1}.$$  
By the associated graded algebra (with respect to $\mathfrak m_{\mathscr A}$) of a graded Artinian algebra $A$, we mean the associated graded of the (ungraded) local algebra $\kappa(A)$, and we will write $\Gr_{\mathfrak m_{A}}(A)$.  The \emph{Hilbert function} of a graded Artinian algebra is the sequence of non-negative integers $H(A)=\left(\dim_{\sf k}(A_0),\dim_{\sf k}(A_1),\ldots,\dim_{\sf k}(A_{j_A})\right)$.  For a local Artinian algebra $\A$, its Hilbert function is that of its associated graded algebra, i.e.  
$$H(\A)=\left(\dim_{\sf k}(\A/\mathfrak{m}_{\A}),\dim_{\sf k}(\mathfrak{m}_{\A}/\mathfrak{m}_{\A}^2),\ldots,\dim_{\sf k}(\mathfrak{m}_{\A}^{j_\A}/\mathfrak{m}_{\A}^{j_\A+1})\right).$$ 
\noindent
For example the non-standard graded ring $A={\sf k}[x]/(x^3)$ with weight ${\sf w}(x)=2$ has Hilbert function $H(A)=(1,0,1,0,1)$, and socle degree $j_A=4$; its localization $\kappa(A)$ has Hilbert function $H(\kappa (A))=(1,1,1)$, and socle degree $j_{\kappa(A)}=2.$\vskip 0.2cm
\noindent
{\it Tensor Products and Free Extensions.}  The tensor product of two graded Artinian algebras is another graded Artinian algebra.  The notion of free extension of Artinian algebras, which generalizes that of tensor product, was introduced by T.~Harima and J.~Watanabe \cite{HW1,HW1e,HW} to study strong Lefschetz properties. See \cite[\S4.2-4.4]{H-W}  and \cite[\S2.1]{IMM2}.
Let $A$, $B$, and $C$ be graded Artinian algebras, with maps $\iota\colon A\rightarrow C$ and $\pi\colon C\rightarrow B$. We say $C$ is a \emph{free extension} with base $A$ and fiber $B$ if $C$ is a free $A$-module via $\iota$, and if $\ker(\pi)=\iota(\mathfrak{m}_A)\cdot C$, the ideal in $C$ generated by the image $\iota(\mathfrak{m}_A)$ (Definition \ref{freeextdef}).
Equivalently, $C$ is a free extension with base $A$ and fiber $B$ if it is isomorphic to the tensor product as $A$-modules (not necessarily as algebras), i.e. $_AC\cong_A(A\otimes_{{\F}}B)$ (\cite[Lemma 2.2]{IMM2}).
\vskip 0.2cm
\noindent
{\it Jordan type and strong Lefschetz.}
In the first part of this paper, we review and study 
the strong Lefschetz and Jordan type properties of tensor products and free extensions: our goal is to better understand the 
relative coinvariant algebra for a pair of finite groups, which is typically a  non-standard graded Artinian Gorenstein base algebra of a free extension.
We will give the precise definitions of Jordan type, strong Lefschetz (SL), and strong Lefschetz Jordan type (SLJT) in Section~\ref{LPJTsec}, but we give a brief overview here.

Recall that for any element $\ell\in\mathfrak{m}_A$, the multiplication map $\times\ell\colon A\rightarrow A$ is a nilpotent linear transformation.  Hence its Jordan canonical form is completely determined by its block sizes, which are encoded in a partition of $\dim_{\sf k}(A)=N$, $P_\ell$, called its \emph{Jordan type}. 
Recall that a partition $P$ of $n$ is a sequence $(p_1,p_2,\ldots,p_r)$ satisfying $p_1\ge p_2\ge \cdots \ge p_r$ and $\sum p_i=n$.  It is convenient to visualize a partition $P$ by its \emph{Ferrers diagram}, which is the arrangement of $n$ dots in $r$ left-justified rows with $p_i$ dots in the $i^{th}$ row (for us, from the top). The \emph{conjugate partition of $P$} is the partition $P^\vee$ whose Ferrers diagram is obtained by switching rows and columns in the Ferrer's diagram of $P$.  For example, the conjugate partition to  $P=(4,4,3,1,1)$ is $P^\vee=(5,3,3,2)$.
The \emph{generic Jordan type} of $A$ is the Jordan type of a generic element $\ell\in\mathfrak{m}_A$; this is the largest Jordan type with respect to the dominance order on partitions.  Besides Jordan types, the \emph{Hilbert function} $H(A)$ defines another partition of $\dim_{\sf k}(A)=n$ whose parts are the dimensions of the graded components of $A$, $\dim_{\sf k}(A_i)$, $i=0,\ldots,j_A$.

Recall that a linear element $\ell\in A_1$ is \emph{strong Lefschetz} if the multiplication maps $\times\ell^k\colon A_i\rightarrow A_{i+k}$ have full rank for all integers $i$ and $k$ (Definition \ref{strongLefdef}). One can show that $\ell\in A_1$ is strong Lefschetz if and only if its Jordan type $P_\ell$ is equal to the conjugate partition of the Hilbert function $H(A)^\vee$ (Proposition \ref{prop:LefschetzJordan}).  We say that a (possibly non-homogeneous) element $\ell\in\mathfrak{m}_A$ has \emph{strong Lefschetz Jordan type} if its Jordan type $P_\ell$ is equal to the Jordan type of a strong Lefschetz element, i.e. $H(A)^\vee$ (Definition \ref{def:SLJT}).  The notion of SLJT was introduced in \cite[Definition~2.23]{IMM1}.

A Jordan basis for an element $\ell\in\mathfrak{m}_A$ with Jordan type $P_\ell=(p_1,\ldots,p_r)$ is a ${\sf k}$ vector space basis for $A$ that can be partitioned into $r$-parts, called \emph{Jordan strings}, of the form $S_i~=~\left\{g_i,\ell g_i,\ldots,\ell^{p_i-1}g_i\right\}$ (Equation \ref{stringeq}).
\begin{example}\label{babyex} Consider the (standard) graded Artinian (complete intersection) algebra $A={\sf k}[x,y]/(x^2,y^3)$ with Hilbert function $H(A)=(1,2,2,1)$.  The conjugate partition is $H(A)^\vee=(4,2)$.  The element $x\in A_1$ has Jordan type $P_x = (2,2,2)$ with Jordan strings $S_1=\left\{1,x\right\}$, $S_2=\left\{y,xy\right\}$, and $S_3=\left\{y^2,xy^2\right\}$.
The generic Jordan type of $A$ is $P=(4,2)=H(A)^\vee$ since the element $\ell=x+y$ has Jordan strings $S_1=\left\{1,\ell,\ell^2,\ell^3\right\}$ and $S_2=\left\{x,\ell x\right\}$.  Thus, $\ell\in A_1$ has SLJT and is a SL element.  \par
However, the non-standard graded algebra ${\sf k}[x]/(x^3)$ with ${\sf w}(x)=2$ is not strong Lefschetz (since $A_1=0$), but the element $x$ has strong Lefschetz Jordan type. 
\end{example}
In a previous paper, three of the present authors showed that the generic Jordan type of a free extension $C$ is at least as large as the generic Jordan type of the tensor product of the base $A$ and the fiber $B$ \cite[Theorem 2.12]{IMM2}. This result can be used to show that if the base and fiber of a free extension both have symmetric Hilbert functions and are strong Lefschetz, then the extension also is strong Lefschetz \cite[Theorem 2.14]{IMM2}, a result due originally to T. Harima and J. Watanabe in characteristic zero \cite[Theorem 6.1]{HW}.
On the other hand, we gave there an example where the free extension $C$ and the fiber $B$ both are SL, but the base $A$ is not SL; in this example $A=R^K_W$ is the relative coinvariant ring for a pair of complex reflection groups $K\subset W\subset\Gl(V)$. We explain this next, and will give a larger class of such examples in Example~\ref{ex:InvEx4}. 
\vskip 0.2cm
\par\noindent
{\it Relative coinvariants.}
We note here that free extensions abound in the invariant theory of finite groups.  For a group $W\subset\Gl(V)$ acting linearly on some vector space $V={\sf k}^n$, the action extends to an action of the group $W$ on the ring of polynomial functions $R=\Sym(V^*)$, and we can consider the subring $R^W\subset R$ of polynomials invariant under that action.  We denote by $\mathfrak{h}_W\subset R$ the ideal generated by $R^W_+$, the invariant polynomials of strictly positive degree (every term has positive degree); the quotient $R/\mathfrak{h}_W$ is the \emph{coinvariant ring} of $W$, denoted by $R_W$.  Given any subgroup $K\subseteq W$, we can take $K$-invariants first, then take $W$-coinvariants to obtain what we call the \emph{relative coinvariant} ring 
$$R^K_W=\frac{R^K}{\mathfrak{h}_W\cap R^K}.$$
In the non-modular case, i.e. $|W|\in{\sf k}^*$, if $K\subseteq W\subset\Gl(V)$ are generated by reflections, then the coinvariant ring $C=R_W$ of $W$ is a free extension with base the relative coinvariant ring $A=R^K_W$ and fiber the coinvariant ring $B=R_K$ of $K$.  If the relative coinvariant ring $A=R^K_W$ and the smaller coinvariant ring $B=R_K$ have symmetic Hilbert functions and are strong Lefschetz, then so is  $R_W$.  On the other hand, there are plenty of pairs of reflection groups $(K,W)$, for which the relative coinvariant ring is not strong Lefschetz.  We prove (Theorem \ref{thm:ParLef})
\begin{uthm}
If $K$ is not a parabolic subgroup of $W$, then $R_W^K$ cannot be strong Lefschetz.
\end{uthm}  We also give several examples of complex reflection groups $K\subset W$ where $K$ is parabolic yet $R^K_W$ still is not SL (we conjecture that this cannot happen for real reflection groups).  On the other hand, we conjecture that for every pair of reflection groups $K\subset W\subset\Gl(V)$, with $K$ parabolic or not, the relative coinvariant ring $R^K_W$ always has SLJT, and we support this conjecture with several examples. \par Finally we describe a connection between relative coinvariant rings and an old conjecture of G. Almkvist.  We show that the generating function for a certain class of restricted partitions studied by Almkvist in the 1980's, are actually Hilbert polynomials of certain relative coinvariant rings $A(m,n)=R^{\mathfrak{S}_n}_{G(m,1,n)}$.  In 1989, Almkvist conjectured that these polynomials ${\sf p}(A(m,n),t)$ have unimodal coefficients for sufficiently large $n$, and we conjecture correspondingly that the rings $A(m,n)$ are strong Lefschetz for sufficiently large $n$, and additionally, that they have elements of SLJT for every $n$.  \par
The goals of this paper are, first, to survey what we know about the Lefschetz and Jordan type properties for free extensions, and to present new results, especially about strong Lefschetz Jordan type, introduced in \cite{IMM1}. Then we apply these results to relative coinvariant rings.

\subsubsection{Overview.}
The organization of this paper is as follows.
In Section \ref{prelimsec} we state relevant definitions and basic results concerning Jordan types and Lefschetz properties, including strong Lefschetz Jordan type. In Section \ref{tensorprodsec} we give the Clebsch-Gordan formula for determining the Jordan type in a tensor product.  We also give a new proof of the well known fact that if the Hilbert functions $H(A)$ and $H(B)$ are symmetric, then SL for $A$ and $B$ implies SL for $C$ (Proposition \ref{prop:JTtensor}). We show, conversely, that if $C$ is SL then so are both $A$ and $B$ (Theorem \ref{thm:TPLef}). We include an example of J.~Watanabe showing that the hypotheses of symmetry of Hilbert functions are necessary for first implication (Example \ref{Watanabeex}). In Section \ref{freeextsec} we first resume properties, give examples of and state open problems about free extensions. 
We state the known result that if the base $A$ and the fiber $B$ are SL and if their Hilbert functions $H(A)$ and $H(B)$ symmetric, then the free extension $C$ is also SL (Proposition \ref{thm:FELef}), but we provide a counterexample
Example~\ref{ex:SNU} showing that the converse is false: the free extension $C$ may be SL, but either the base $A$ or the fiber $B$ may fail to be SL.
\par
In Section \ref{invthsec3.1} we study rings  $A=R^K_W$ of relative coinvariants, and in particular the properties of unimodality of Hilbert function and Jordan type. In Theorem~\ref{thm:ParLef}
we show that if a reflection subgroup $K$ of a non-modular finite reflection group over a field $\sf k$  is not parabolic, then the relative coinvariant ring $R^K_W$ cannot have the strong Lefschetz property. We then give examples where the relative coinvariant ring $A$ is not strong Lefschetz, but has non-homogeneous elements of strong Lefschetz Jordan type (Proposition \ref{Gmmnprop}). We study in Section \ref{A(m,n)sec} strong Lefschetz Jordan type for several infinite sequences $A(m,n)=R^{\mathfrak S_n}_{G(m,1,n)}$ and $A(m,p,n)$ of relative coinvariant algebras defined from the Shephard-Todd classification of complex reflection groups. Some open problems we pose concerning the Lefschetz properties of the algebras $A(m,n)$ are related to combinatorial problems concerning the unimodality of certain partitions functions, posed and partially solved by G.~Almkvist in 1989.

\tableofcontents
\section{Properties of tensor products and free extensions.}\label{prelimsec}
\subsection{Lefschetz Properties and Jordan Type.}\label{LPJTsec} 
In this section we give some basic definitions and facts.  Some proofs are omitted but we give references. Graded algebras are not necessarily standard-graded.
\begin{definition}\label{strongLefdef}\cite[Definition 3.8]{H-W}
Let $A$ be a graded Artinian algebra.  We say that a linear form $\ell\in A_1$ is \emph{strong Lefschetz} if for every pair of integers $i$ and $k$, the multiplication maps 
$$\times \ell^k\colon A_i\rightarrow A_{i+k}$$
have maximal rank, i.e. $\min\left\{\dim(A_i),\dim(A_{i+k})\right\}$. 
\end{definition}.
\begin{remark}\label{2.2rem}
	T. Harima and J. Watanabe \cite{HW} and \cite[Definition 3.18 and Theorem 3.22 ]{H-W}  term the above notion \emph{strong Lefschetz in the general sense}; in the related notion of \emph{strong Lefschetz in the narrow sense} one requires that the multiplication maps $\times\ell^{j_A-2i}\colon A_i\rightarrow A_{j_A-i}$ are isomorphisms for each $0\leq i\leq \left\lfloor\frac{j_A}{2}\right\rfloor$.  When $A$ has a symmetric Hilbert function, strong Lefschetz in the narrow sense and strong Lefschetz in the sense of Definition \ref{strongLefdef} are equivalent.  Indeed, if $\times\ell^{j_A-2i}\colon A_i\rightarrow A_{j_A-i}$ is an isomorphism for each $i$, then for any pair $i,k$, the multiplication map $\times\ell^k\colon\rightarrow A_i\rightarrow A_{i+k}$ is either injective if $i+k<j_A-i$ or is surjective if $i+k\geq j_A-i$.  
\end{remark}\par
As mentioned earlier, an element $\ell\in \mathfrak m_A$ (possibly non-homogeneous) is nilpotent, so its Jordan canonical form is determined by the partition $P_\ell$ of $n=\dim_{\sf k}(A)$ giving its block sizes,
\begin{equation}\label{Jordaneq}
P_\ell=(p_1,\ldots, p_r), \text { where } p_1\geq \cdots\geq p_r,
\end{equation}
which we write for short $P_\ell=(p_1\geq\cdots \geq p_r)$. 
We will refer to a \emph{Jordan string} of an element $\ell\in\mathfrak{m}_A$, which are a sequence $S_i$ of elements of $A$,
\begin{equation}\label{stringeq}
S_i=\left\{z_i,\ell\cdot z_i,\ldots,\ell^{p_i-1}\cdot z_i\right\}
\end{equation}
 that are the part of a Jordan basis for the multiplication map $\times \ell\colon A\rightarrow A$ corresponding to the Jordan block indexed by the part $p_i$ in $P_\ell$.  The strings themselves are not unique given $(\ell,A)$, but their cardinalities are.
  
We will regard the Hilbert function of $A$, written $H(A)$, as defining a partition $P(H)$ of $n=\dim_{F}(A)$.  The \emph{conjugate partition} of $P=(p_1\ge\cdots \ge p_r)$ is 
\begin{equation}\label{HAconjeq}
P^\vee=(m_1\geq \cdots \geq m_t), m_i=\#\left\{j\, \big | \,p_j\geq i.\right\}.
\end{equation}
Recall that the Ferrers diagram of $P^\vee$ is obtained by switching rows and columns of that of $P$.\par

Given two partitions $P=(p_1\geq \cdots\geq p_r)$ and $Q=(q_1\geq\cdots\geq q_s)$ of the same number $n$, we say that $Q$ \emph{dominates} $P$, and write $P\leq Q$, if 
\begin{equation}\label{dominanceeq}
p_1+\cdots+p_i\leq q_1+\cdots+q_i \ \ \text { for all } \ i.
\end{equation}
This defines the \emph{dominance partial order} on the partitions of $n$.

Recall that the \emph{generic Jordan type} of an Artinian algebra is the largest Jordan type $P_\ell$, with respect to the dominance order, that occurs among all elements $\ell\in\mathfrak{m}_A$.  The following result relates Lefschetz properties and Jordan types of homogeneous elements in the graded case: it is shown in \cite[Theorem 2.19ii, and Proposition 2.32]{IMM1} (see \cite[Proposition~3.64]{H-W} for the case $A$ standard graded).

\begin{proposition}
	\label{prop:LefschetzJordan}
	Let $A$ be a graded Artinian algebra, and let $\ell\in\mathfrak{m}_A$ be a homogeneous element.  Then in the dominance order we have 
	$$P_\ell\leq H(A)^\vee.$$
	Moreover $\ell\in\mathfrak{m}_A$ is strong Lefschetz in the general sense (Definition \ref{strongLefdef}) if and only if $\ell\in A_1$ and  
	$$P_\ell=H(A)^\vee.$$
\end{proposition}   
\begin{definition}
	\label{def:SLJT}
	An element $\ell\in\mathfrak{m}_A$ (possibly non-homogeneous) is said to have \emph{strong Lefschetz Jordan type} if its Jordan type is equal to the conjugate partition of $A$, i.e. 
	$$P_\ell= H(A)^\vee.$$
\end{definition}

As one might expect, the relationship between Jordan type and the Hilbert function $H(A)$ is not so obvious for non-homogeneous elements.  For example, the following question appears to be open for non-standard graded Artinian algebras.
\begin{question}
	\label{ques:NonHomBigger}
	 Can a graded Artinian algebra $A$ have a (non-homogeneous) element $\ell\in\mathfrak{m}_A$ such that $P_\ell>H(A)^\vee$? 
\end{question}
On the other hand, there is a comparison between the Jordan type of an arbitrary element $\ell\in\mathfrak{m}_A$ and the Hilbert function of its associated graded algebra $\Gr_{\mathfrak{m}_A}(A)$. 

The following result is shown in the appendix of \cite{IMM1}.
\begin{proposition}\cite[Theorem 2.19i$'$]{IMM1}, 
	\label{prop:AssGr}
	Let $A$ be a graded algebra and let $\ell\in\mathfrak{m}_A$ be any (possibly non-homogeneous) element. Then
	$$P_\ell\leq H(\Gr_{\mathfrak m_A}(A))^\vee$$
	in the dominance order.
\end{proposition}

\begin{remark}
If $A$ has the standard grading, one can show that $A\cong \operatorname{Gr}_{\mathfrak{m}_A}(A)$ as graded algebras, and hence $H(A)=H(\operatorname{Gr}_{\mathfrak{m}_A}(A))$: thus Question \ref{ques:NonHomBigger} has the answer ``No'' when $A$ is standard graded. On the other hand, if $A$ has a non-standard grading then this equality no longer holds.  In the dominance order, the conjugate partition $H(A)^\vee$ may be greater than, less than, or equal to $H(\Gr_{\mathfrak m_A}(A))^\vee$, as the following examples show.
\end{remark}  

\begin{example}
	\label{ex:leq}
	Let $A={\sf k}[x,y]/(x^2,y^2)$ with weights ${\sf w}(x,y)=(1,2)$ so that $H(A)=(1,1,1,1)$ and $H(A)^\vee=(4)$; then $\Gr_{\mathfrak m_A}(A)={\sf k}[x,y]/(x^2,y^2)$ has weights ${\sf w}(x,y)=(1,1)$ with Hilbert function $H(\Gr_{\mathfrak m_A}(A))=(1,2,1)$ whose conjugate partition $H(\Gr_{\mathfrak m_A}(A))^\vee=(3,1)<(4)=H(A)^\vee$.  Note that for $\ell=x+y$ we have $P_{\ell,A}=(3,1)$, which is the generic Jordan type of $A$ and is less than $(4)$ in the dominance partial order, hence $A$ cannot have a SLJT element.  In particular, $A$ does not have the strong Lefschetz property.
\end{example}

\begin{example}
	\label{ex:geq}
	Take $B={\sf k}[x,y,z]/(xz-y^3,yz,z^2,x^4y,x^5)$ with weights ${\sf w}(x,y,z)=(1,1,2)$, so that $H(B)=(1,2,4,4,4,2,1)$ and $H(B)^\vee=(7,5,3,3)$. Then $\Gr_{\mathfrak m_B}(B)={\sf k}[x,y,z]/(x^5,y^4,z^2,xz,yz,x^4y)$ which has weights ${\sf w}(x,y,z)=(1,1,1)$, with Hilbert function $H(\Gr_{\mathfrak m_B}(B))=(1,3,3,4,4,2,1)$ whose conjugate partition $H(\Gr_{\mathfrak m_B}(B))^\vee=(7,5,4,2)$ is greater than $(7,5,3,3)=H(B)^\vee$.  The non-homogeneous element $\ell'=x+y+z\in\mathfrak{m}_B$ has Jordan type $P_{\ell',B}=(7,5,3,3)=H(B)^\vee$ (calculation in {\sc Macaulay~2}), hence $\ell'$ has SLJT in $B$. But $B$ does not have a linear element of SLJT so is not SL.
\end{example}
  The following is an immediate consequence of Proposition \ref{prop:AssGr}.
  \begin{lemma}
  	\label{lem:AssGrLef}
  	For a non-standard graded Artinian algebra $A$ with associated graded Algebra $\Gr_{\mathfrak m_A}(A)$, if the Hilbert function $H(\Gr_{\mathfrak m_A}(A))^\vee $ is less than $H(A)^\vee$ in the dominance order, then $A$ cannot be strong Lefschetz.
  \end{lemma}
  Note that an example of an algebra $A$ answering the Question \ref{ques:NonHomBigger} above positively, must satisfy the inequality $H(\Gr_{\mathfrak m_A}(A))^\vee> H(A)^\vee$. 

Recall that a sequence of integers $(h_0,h_1,\ldots,h_j)$ is unimodal if there is some index $c$ for which $h_0\leq\cdots\leq h_c\geq h_{c+1}\geq\cdots h_j$.  A proof of the following is in \cite[Proposition 2.36]{IMM1}.
\begin{proposition}
	\label{prop:UnimodalSLJTSL}
	If $A$ is a graded Artinian algebra with the standard grading and unimodal Hilbert function, then $A$ has an element $\ell\in\mathfrak{m}_A$ of strong Lefschetz Jordan type if and only if it has a strong Lefschetz element $\ell'\in A_1$.
\end{proposition} 
We believe that the unimodal condition in Proposition \ref{prop:UnimodalSLJTSL} is necessary, but we do not have an example to show this.

To summarize, having an element $\ell\in\mathfrak{m}_A$  of strong Lefschetz Jordan type is equivalent to the usual strong Lefschetz condition of Definition \ref{strongLefdef} for a standard graded Artinian algebra $A$ whose Hilbert function is unimodal, but having a SLJT element may be strictly weaker than $A$ being SL for Artinian algebras with non-standard gradings or with non-unimodal Hilbert functions.

The following result is well known (and has been reproved several times). 
\begin{lemma}
	\label{heighttwolem} 
	Let 
	$\mathscr{A}={\sf k}\{x,y\}/I$ be a local Artinian algebra of codimension two and socle degree $j_\A$,  and suppose $\cha {\sf k}=0$ or $\cha {\sf k}\ge j_{\mathscr{A}}$. Then the generic Jordan type of $\mathscr{A}$ is equal to the conjugate partition of its Hilbert function.
	In particular, if $A$ is a graded Artinian algebra of codimension two, and if $H(A)^\vee=H(\Gr_{\mathfrak m_{A}}(A))^\vee$ then $A$ must have a SLJT element.
\end{lemma}
\begin{proof} These statements follow from a standard basis argument of J. Brian\c{c}on in his 1977 article \cite{Bri} for ideals in $\mathbb C[x,y]$, that extends readily to the case $\cha {\sf k}=p\ge j$ (see \cite[Theorem 2.16]{BasI} for a discussion, and \cite[Proposition 3.15]{H-W} for a proof in $\cha{\sf k}=0$).\footnote{Reproofs for the standard graded case were essentially the same as J. Brian\c{c}on's, see also \cite[Theorem 2.7]{MiNa}.}  The second statement follows from $\mathscr A=\kappa(A)$ being strong Lefschetz, implying that an 
element $\ell\in \mathfrak m_{\mathscr{A}}$ has SLJT: so the pre-image -- the same $\ell\in \mathfrak m_A$ -- has SlJT.
\end{proof}\par

The following gives a useful criterion for applying Lemma \ref{heighttwolem}.  The proof is an easy exercise.
\begin{lemma}
	\label{lem:heighttwoHF}
	For positive integers $a,b,m,n\in\Z$ satisfying $a\leq b$ and $am=bn$, the graded Artinian algebra 
	$$A=\frac{{\sf k}[x,y]}{(x^a-y^b,xy)}, \ w(x,y)=(m,n)$$ 
	satisfies $H(A)^\vee=H(\Gr_{\mathfrak m_{A}}(A))^\vee$ if and only if $n|m$ and $(a-1)m=bn$. 
\end{lemma}
\subsection{Tensor Products.}\label{tensorprodsec}
If $A$ and $B$ are graded Artinian algebras then so is their tensor product $C=A\otimes_{{\F}}B$.  We will relate Jordan types in $A$ and $B$ to Jordan types in $C$.  Consider first the simple case ${\sf k}[x]/(x^m)\otimes {\sf k}[y]/(y^n)$ which is isomorphic to the algebra $R(m,n)= {\sf k}[x,y]/(x^m,y^n)$.  We assume the grading is standard so that ${\sf w}(x,y)=(1,1)$.  By Lemma \ref{heighttwolem} if $\cha {\sf k}=0$ or is at least $m+n-1$, then $R(m,n)$ has the strong Lefschetz property, with strong Lefschetz element $\ell=ax+by$ that may be scaled to $\ell=x+y$. By Proposition~\ref{prop:LefschetzJordan}, we have $P_\ell=H(R(m,n))^\vee$.  This is essentially a ring-theoretic version of the well known Clebsch-Gordan decomposition for tensor products of irreducible $\mathfrak{sl}_2$ representations.  See also \cite[Lemma 3.70]{H-W}.

\begin{lemma}[Clebsch-Gordan]\label{Clebschlem}
	Assume $\cha {\sf k} =0$ or $\cha{\sf k}> m+n-1$, and $n\ge m$. Then with $R(m,n)$ and $\ell=x+y$ as above, we have
	\begin{align} H(R(m,n))=&(1,2,3,\ldots,m-1,\underbrace{m,m,\ldots,m}_{n-m+1},m-1,\ldots, 3,2,1),\text { and }\notag\\
	P_\ell=& H(R(m,n))^\vee= (n+m-1,n+m-3,\ldots,n-m+1).\label{ClebschGeq}
	\end{align}
\end{lemma}
If we denote by $[m]$ the Jordan type of the multiplication map $\times x\colon {\sf k}[x]/(x^m)\rightarrow {\sf k}[x]/(x^m)$, then we can write the Clebsch-Gordan formula \eqref{ClebschGeq} as
\begin{equation}
\label{eq:CGeqShort}
[m]\otimes [n]=\bigoplus_{k=0}^{\min\{m,n\}}[n+m-2k+1].
\end{equation}
Using formula \eqref{eq:CGeqShort} and that tensor products distribute over direct sums, one can readily prove the following (see \cite[Proposition 3.66]{H-W}).
\begin{proposition}
	\label{prop:JTtensor}
	Suppose that $A$ and $B$ are (possibly non-standard) graded Artinian algebras with linear forms $\ell_A\in A_1$ and $\ell_B\in B_1$ and Jordan types\\ $P_{\ell_A}=\bigoplus_{i=1}^r[p_i]$ and $P_{\ell_B}=\bigoplus_{j=1}^s[q_j].$  
	Then in the tensor product $C=A\otimes_{\sf k}B$, the linear form $\ell_C=\ell_A\otimes 1+1\otimes\ell_B\in C_1$ has Jordan type 
	\begin{equation}
	\label{eq:JTtensor}
	P_{\ell_C}=\bigoplus_{i=1}^r\bigoplus_{j=1}^s\bigoplus_{k=0}^{\min\{p_i,q_j\}}[p_i+q_j-2k+1].
	\end{equation}
\end{proposition}

For a graded Artinian algebra $A$ and a linear form $\ell\in A_1$ with Jordan type $P_\ell=(p_1,\ldots,p_r)=\bigoplus_{i=1}^r[p_i]$, we have a decomposition of $A$ into graded cyclic ${\sf k}[\ell]$-modules:
\begin{equation}
\label{eq:kldecomp}
A\cong \frac{{\sf k}[\ell]}{(\ell^{p_1})}(a_1)\oplus\cdots\oplus\frac{{\sf k}[\ell]}{(\ell^{p_r})}(a_r).
\end{equation}
where the $i^{th}$ summand $\frac{{\sf k}[\ell]}{(\ell^{p_i})}(a_i)$ is the ${\sf k}$-span of the Jordan string $\langle S_i\rangle$ having length $p_i$ and a cyclic generator in degree $a_i$.  The following is well known, but we give a quick sketch of a proof here.  For another proof of this and related results, see \cite[Theorem 3.3]{Lindsey}.
\begin{lemma}
	\label{lem:LefSymm} Assume that $A$ is graded Artinian (not necessarily standard-graded) with symmetric Hilbert function.
	The linear form $\ell\in A_1$ is strong Lefschetz for $A$ if and only if each summand in the decomposition \eqref{eq:kldecomp} is centered, meaning that for each $1\leq i\leq r$ we have 
	$$\frac{j_A}{2}=\frac{2a_i+p_i-1}{2} \ \text{(so} \ p_i=j_A+1-2a_i).$$
	
	Then $A$ is narrow strong Lefschetz, in the sense of Remark \ref{2.2rem}.
\end{lemma}
\begin{proof}
	Assume first that every Jordan string is centered.  Then for each degree $i$, every Jordan string containing an element from $A_i$ must also contain a non-zero element from $A_{j_A-i}$.  This implies that for each degree $i$, the map $\times\ell^{j_A-2i}\colon A_i\rightarrow A_{j_A-i}$ is at least injective, and hence bijective by the symmetry of the Hilbert function.
	
	Conversely, assume that $\ell\in A_1$ is strong Lefschetz.  Consider the $i^{th}$ Jordan string $S_i=\left\{z_i,\ell z_i,\ldots,\ell^{p_i-1}z_i\right\}$ where $\deg(z_i)=a_i$.  Then $\ell^{p_i-1}\cdot z_i\in A_{a_i+p_i-1}$.  Since $\ell$ is strong Lefschetz, by Remark \ref{2.2rem} the multiplication map $$\times\ell^{j_A-2(a_i+p_i-1)}\colon A_{a_i+p_i-1}\rightarrow A_{j_A-(a_i+p_i-1)}$$
	is an isomorphism, hence we must have $\ell^{j_A-2(a_i+p_i-1)}\cdot\ell^{p_i-1}\cdot z_i=\ell^{j_A-2a_i-p_1+1}z_i$
	which implies that $j_A-2a_i-p_i+1\geq 0$.  Similarly, we also know that $\ell^{j_A-2a_i}\cdot z_i\neq 0$, which implies that $j_A-2a_i\leq p_i-1$.  Putting these two inequalities together we obtain the desired equality $j_A=2a_i+p_i-1$.
\end{proof}

\begin{corollary}
	\label{cor:TPLef}
	Assume that $\cha{\sf k}=0$ or $\cha{\sf k}>j_A+j_B$.  Suppose $A$ and $B$ are graded Artinian algebras with symmetric Hilbert functions.  If $x\in A_1$ and $y\in B_1$ are strong Lefschetz elements for $A$ and $B$, respectively, then the element $\ell=x\otimes 1+1\otimes y\in C_1$ is strong Lefschetz for the tensor product algebra $C=A\otimes_{\sf k}B$.
\end{corollary}
\begin{proof}
	Suppose that $A$ and $B$ have Lefschetz decompositions 
	$$A\cong\bigoplus_{i=1}^r\frac{{\sf k}[x]}{(x^{p_i})}(a_i) \ \ \text{and} \ \ B\cong\bigoplus_{j=1}^s\frac{{\sf k}[y]}{(y^{q_j})}(b_j).$$
	Then by the Clebsch-Gordan formula and Proposition \ref{prop:JTtensor}, $C$ has ${\sf k}[\ell]$-module decomposition
	$$C\cong \bigoplus_{i=1}^r\bigoplus_{j=1}^s\bigoplus_{k=0}^{\min\{p_i,q_j\}}\frac{{\sf k}[\ell]}{(\ell^{p_i+q_j-2k+1})}(a_i+b_j+k-1).$$
	If $j_A$ and $j_B$ are the socle degrees of $A$ and $B$, then $j_C=j_A+j_B$ is the socle degree of $C$, and it is straightforward to verify to check that for each index $(i,j)$ we have 
	$$2(a_i+b_j+k-1)+(p_i+q_j-2k+1)-1=(2a_i+p_i-1)+(2b_j+q_j-1)=j_A+j_B=j_C.$$
	Hence by the criterion of Lemma \ref{lem:LefSymm}, $\ell\in C_1$ is strong Lefschetz for $C$.
\end{proof}\par
A slightly stronger result is proved in the paper \cite{Lindsey}.  See also \cite[Theorem~3.34]{H-W}, or \cite[Proposition 3.10]{HW} for a related result.\footnote{In \cite{HW}, T. Harima and J. Watanabe show the following: assume that $\cha {\sf k}=0$, that $V,W$ are $A,B$ modules with symmetric unimodal Hilbert functions; then $V\otimes_{\sf k} W$ has the SLP property as $A\otimes_{\sf k}B$ module if and only if both $V,W$ are SLP modules.}
\vskip 0.15cm
The \emph{Hilbert polynomial} ${\sf p}(A, t)$ of a graded Artinian algebra $A$ with Hilbert function $H(A)$ is
\begin{equation}
{\sf p}(A, t)=\sum_{i\ge 0}^{j_A} H(A)_it^i,
\end{equation}
where $H(A)_i=\dim_{\F}A_i$.\footnote{The Hilbert polynomial is sometimes termed ``Poincar\'{e} polynomial'' or ``Poincar\'{e} series'' by topologists or topology-influenced writers \cite{MeSm,SmSt-PBI}; but ``Poincar\'{e} series '' has a very different meaning to commutative algebraists, so we use ``Hilbert polynomial'' in this paper.} 
Given two Hilbert functions $H=(h_0=1,h_1,\ldots,h_m)$ and $H'=(h'_0=1,h'_1,\ldots,h'_n)$, define their tensor product $H\otimes H'=H''=(h''_0=1,h''_1,\ldots,h''_{m+n})$ by the formula
\begin{equation}\label{tensorprodHeq}
\sum_{i=0}^{m+n}h_i''t^i={\sf p}(H'',t)={\sf p}(H,t)\cdot{\sf p}(H',t).
\end{equation}
\par
The next example shows that the symmetric Hilbert function condition in Corollary \ref{cor:TPLef} cannot be dropped.  We thank J. Watanabe for showing us this example.  Further examples of this sort can be found in \cite[Theorem 4.2]{MNS}.
\begin{example}[J. Watanabe]\label{Watanabeex}
	Let $A={\F}[x,y]/(x^2,xy,y^4)$ with the standard grading and let $B={\F}[w,z]/(w^2,wz,z^4)$ with grading ${\sf w}(w,z)=(2,1)$; thus $H(A)=(1,2,1,1)$ and $H(B)=(1,1,2,1)$ and $H(A)^\vee=(4,1)=H(B)^\vee$.  Let $C=A\otimes_{\F}B$ so that $C_1=\langle (x,y,z)\rangle $ and $H(C)=H(A)\otimes H(B)=(1,3,5,7,5,3,1)$, hence $H(C)^\vee=(7,5,5,3,3,1,1) $. A generic element of $C_1$ up to scaling is $\ell=(x+y)+z$ where $(x+y)\in A_1, z\in B_1$, are strong Lefschetz for $A$ and $B$, hence their Jordan types are $P_(x+y)=(4,1)=[4]+[1]=P_z$.  By the Clebsch-Gordan formula \eqref{ClebschGeq} $\ell$ has Jordan type 
	\begin{align*}
	P_\ell=([4]+[1])\otimes ([4]+[1])&=[4]\otimes [4]+[4]\otimes [1]+[1]\otimes [4]+[1]\otimes [1]\\
	&=(7,5,3,1)+[4]+[4]+[1]\\
	&=(7,5,4,4,3,1)\neq H(C)^\vee.
	\end{align*}
	Hence $\ell$ is not strong Lefschetz for $C$.
\end{example}

Finally we prove a converse to Corollary \ref{cor:TPLef}.  This was also proven by T.~Harima and J. Watanabe in \cite[Theorem 3.10]{HW} under the assumptions that the Hilbert functions for $A$ and $B$ are both symmetric \emph{and} unimodal.  Our proof is inspired by E. Babson and E. Nevo's proof of \cite[Lemma 4.3]{BN}: there they use a trick of shifting the ``strong Lefschetz defect'' to the middle degrees by tensoring with a simple algebra.  The virtue of our proof is that we do not need to assume the Hilbert functions are unimodal \emph{a priori}.  Of course unimodality of the Hilbert function is \emph{a posteriori} a consequence of the strong Lefschetz property for a graded Artinian algebra\footnote{Indeed if $H(A)$ is not unimodal, then there must be indices $a<b<c$ such that $H(A)_{a}>H(A)_b<H(A)_{c}$, which implies that the the map $\ell^{c-a}\colon A_{a}\rightarrow A_{c}$ has rank at most $H(A)_b$, and thus cannot have full rank.}. 

\begin{theorem}
	\label{thm:TPLef}
	Let $A$ and $B$ be graded Artinian algebras with socle degrees $j_A=a$ and $j_B=b$, and let $C=A\otimes_{F} B=\oplus_{i=0}^cC_i$ be their tensor product algebra (note $j_C=a+b$).  Assume that $A$ and $B$ have symmetric Hilbert functions and that $\cha{\F}=0$ or $\cha{\F}>j_A+j_B$. Then $\ell_C=\ell_A\otimes 1+1\otimes \ell_B\in C_1$ is strong Lefschetz for $C$ implies that $\ell_A\in A_1$ is strong Lefschetz for $A$ and $\ell_B\in B_1$ is strong Lefschetz for $B$.
\end{theorem}
\begin{proof}
	Assume that the tensor product algebra $C=A\otimes_{{\F}}B$ has the strong Lefschetz property, and let $\ell_C=\ell_A\otimes 1+1\otimes\ell_B$ be a strong Lefschetz element for $C$.  It suffices to show that $\ell_A\in A_1$ is strong Lefschetz for $A$, the argument for $B$ being analogous.
	By way of contradiction, assume that $\ell_A$ is not a strong Lefschetz element for $A$.  First assume that the socle degree $j_A=2m+1$ is odd,  and that strong Lefschetz for $\ell=\ell_A$ fails in the middle degree, i.e.
	$$\times\ell\colon A_{m}\rightarrow A_{m+1}$$
	is not an isomorphism.
	Since the Hilbert function is symmetric, this implies that the map is not injective, hence there is some non-zero element $\alpha\in A_m$ such that $\ell\cdot\alpha=0$.  Set $\gamma=\alpha\otimes 1\in C_m$.  Note that $\gamma$ is non-zero since $\alpha$ is non-zero.  Then we have 
	\begin{align*}
	\ell_C^{c-2m}\cdot \gamma= \ & \left(\ell_A\otimes 1+1\otimes\ell_B\right)^{b+1}\cdot(\alpha\otimes 1)\\
	=\  & \alpha\otimes\ell_B^{b+1}=0
	\end{align*}
	which contradicts our assumption that $\ell_C\in C_1$ is strong Lefschetz for $C$.
	
	Next we suppose that $j_A=a$ is arbitrary, and let $i$ with $0\leq i\leq \left\lfloor\frac{a}{2}\right\rfloor$ be the largest index for which the multiplication map $\times\ell^{a-2i}\colon A_i\rightarrow A_{a-i}$ is not an isomorphism (hence injective).  Then there must be an element $\alpha\in A_i$ for which $\ell^{a-2i}\cdot\alpha=0$.  Define the new algebra $$A'=A\otimes_{\F}{\F}[t]/(t^{a-2(i-1)})=A[t]/(t^{a-2(i-1)}).$$
	Note that the socle degree of $A'$ is $a'=2(a-i)+1$.  Also note that strong Lefschetz for the linear form $\ell'=\ell\otimes 1+1\otimes t\in A'_1$ fails in the middle degree.  To see this consider the element $\beta\in A'_{a-i}$ given by 
	$$\beta=\ell^{a-2i-1}\alpha\otimes t^2-\ell^{a-2i-2}\alpha\otimes t^3+\cdots+(-1)^{a-2i}\alpha\otimes t^{a-2i+1}.$$
	Note that $\beta$ is non-zero since we are assuming that $\alpha$ is non-zero.
	Then we have 
	\begin{align*}
	\ell'\cdot\beta= & \left(\ell\otimes 1+1\otimes t\right)\cdot\left(\ell^{a-2i-1}\alpha\otimes t^2-\ell^{a-2i-2}\alpha\otimes t^3+\cdots+(-1)^{a-2i}\alpha\otimes t^{a-2i+1}\right)\\
	&(\text{the sum telescopes})\\
	= \, & \ell^{a-2i}\cdot\alpha\otimes t^2+(-1)^{a-2i}\alpha\otimes t^{a-2i+2}=0
	\end{align*}
	Note that $C'=A'\otimes_{{\F}}B\cong C\otimes_{{\F}} {\F}[t]/(t^{a-2i+2})$.  Since $\ell_C$ is strong Lefschetz for $C$ and $t$ is strong Lefschetz for ${\F}[t]/(t^{a-2i+2})$, we must have $\ell_{C'}=\ell_C\otimes 1+1\otimes t=\ell_{A'}\otimes 1+1\otimes\ell_B$ is strong Lefschetz for $C'$.  On the other hand, the argument above shows that $\ell_{C'}$ cannot possibly be strong Lefschetz for $C'$ since $a'$ is odd and $\ell_{A'}$ fails strong Lefschetz property in the middle degree.  This shows our assumption that $\ell $ is not SL for $A$ is false, and completes the proof of the Theorem.
\end{proof}

It is an open problem whether the result of Theorem \ref{thm:TPLef} holds without the symmetric Hilbert function hypotheses.  It is also open as to what can be said at all regarding SLJT and tensor products.

\subsection{Free Extensions.}\label{freeextsec}
The notion of free extension generalizes that of a tensor product, \cite[\S4.2-4.4]{H-W}  and \cite[\S2.1]{IMM2}.  Free extensions have also been studied in the topology literature, e.g. \cite{MoSm,SS,SmSt-PBI}, in terms of coexact sequences.

\begin{definition}\label{freeextdef} Let $A$, $B$, and $C$ be graded Artinian algebras, with maps $\iota\colon A\rightarrow C$ and $\pi\colon C\rightarrow B$. 
The graded algebra	$C$ is a \emph{free extension} of the base $A$ with fiber $B$ if both
	\begin{enumerate}[i.]
		\item $\iota\colon A\rightarrow C$ makes $C$ into a free $A$-module, and
		\item $\pi\colon C\rightarrow B$ is surjective with $\ker(\pi)=(\iota(A)_+)\cdot C$.
	\end{enumerate}
 \end{definition}
Note that if $C=A\otimes_{\sf k}B$, the inclusion $\iota\colon A\rightarrow C$, $\iota(a)=a\otimes 1$ makes $C$ into a free $A$-module, and the natural projection $\pi\colon C\rightarrow B$ has kernel $\ker(\pi)=A^+\cdot C$.  Hence the tensor product $C=A\otimes_{\sf k} B$ is a free extension of $A$ with fiber $B$, and also a free extension of $B$ with fiber $A$.  One may regard a free extension as a tensor product in which one of the factors has been deformed.  See \cite[Theorem~2.12]{IMM2} or \cite[Theorem 1]{McD} for more on this perspective. 
\par  A sequence
\begin{equation}\label{coexacteq}
	\xymatrix{{\sf k}\ar[r] & A\ar[r]^-\iota & C\ar[r]^-\pi & B\ar[r] & {\sf k}}
	\end{equation}
	is \emph{coexact}  if $\pi$ is surjective and $\ker(\pi)=(\iota ({\mathfrak m_A}))C$. The following result \cite[Lemma 2.2]{IMM2} gives a useful criterion for identifying free extensions. 
\begin{lemma}\label{coexactfreelem} 
	Let $A, B, C$ be graded Artinian algebras with maps $\iota\colon A\rightarrow C$ and $\pi\colon C\rightarrow B$ and suppose that $\pi$ is surjective.  Then 
	the following are equivalent.
		\begin{enumerate}[(i).]
			\item For every ${\sf k}$-linear section ${\sf s}\colon B\rightarrow C$ of $\pi$, the map $\Phi_{\sf s}=\iota\otimes{\sf s}\colon _A\left(A\otimes_{\sf k}B\right)\rightarrow_AC$ is an isomorphism of $A$-modules, i.e. $C$ is an $A$-module tensor product.
			\item The sequence $
	\xymatrix{{\sf k}\ar[r] & A\ar[r]^-\iota & C\ar[r]^-\pi & B\ar[r] & {\sf k}}
	$
			 is coexact and $\iota\colon A\rightarrow C$ is a free extension.
			\item $\iota\colon A\rightarrow C$ is a free extension and $\ker(\pi)=(\iota(\mathfrak{m}_A))\cdot C$.
			\item $\ker(\pi)=(\iota(\mathfrak{m}_A))\cdot C$ and $\dim_{\sf k}(C)=\dim_{\sf k}(A)\cdot\dim_{\sf k}(B)$.
		\end{enumerate}
\end{lemma}

The next result was originally shown by T. Harima and J. Watanabe \cite[Theorem 6.1]{HW} using their theory of central simple modules.\footnote{Proposition \ref{thm:FELef} was first established by T. Harima and J. Watanabe in their paper \cite{HW1,HW1e} for standard grading; they extended is to non-standard grading over $\cha {\sf K}=0$ in \cite[Proposition~6.1]{HW} (also \cite[Proposition 4.12]{H-W})
where $\cha \F=0$, but the proof in large enough characteristic $p$
	can be shown similarly. }  It is also proved in \cite[Theorem 2.14]{IMM2} using Corollary \ref{cor:TPLef} and a deformation argument. 
\begin{proposition}\label{thm:FELef}
	Assume that the graded Artinian algebra $C$ is a free extension with base $A$ and fiber $B$. Assume also that $\cha \F=0$ or $\cha{\F}>j_A+j_B$ and that the Hilbert functions of both $A$ and $B$ are symmetric. If both $A$ and $B$ have the strong Lefschetz property, then so does $C$.
\end{proposition}

Next we observe that the converse of Proposition \ref{thm:FELef} does not hold in general, in contrast with the converse of Corollary \ref{cor:TPLef} for tensor products that we showed in Theorem \ref{thm:TPLef}.   In the next example the fiber $B$ and the extension $C$ are both strong Lefschetz, but the base $A$ is not.  On the other hand, each of $A$, $B$, and $C$ \emph{do} have elements of SLJT.
Such examples occur rather frequently in invariant theory, as we will see in Section \ref{invthsec}.

\begin{example}\label{ex:SNU}\footnote{This has the same Hilbert functions as the (different) Example
\ref{ex:InvEx4}.}
	Let $R={\F}[x_1,x_2,x_3]$ be the polynomial ring in three variables with standard grading.  Let $e_1,e_2,e_3\in R$ be the elementary symmetric polynomials, i.e. $e_1=x_1+x_2+x_3, e_2=x_1x_2+x_1x_3+x_2x_3, e_3=x_1x_2x_3$, and let $\hat{e}_i=e_i(x_1^3,x_2^3,x_3^3)\in R$ be the symmetric polynomials in the cubed variables.  We have the relations
	\begin{align*}
	\hat{e}_1= & e_1^3-3e_1e_2+3e_3\\
	\hat{e}_2= & e_2^3-3e_1e_2e_3+3e_3^2\\
	\hat{e}_3= & e_3^3.
	\end{align*}
	Define $C=R/(\hat{e}_1,\hat{e}_2,e_3)$ and $B=R/(e_1,e_2,e_3)$ and let $\pi\colon C\rightarrow B$ be the natural projection map. In the next section we will see that $B$ and $C$ are coinvariant rings $R_K$ and $R_W$ for certain reflection groups $K\subset W$ acting on $R$.
	Taking new variables $z_1,z_2,z_3$ where  $\deg(z_i)=i$, we define
	 \begin{align*}
	A= &\  {\F}\ [z_1,z_2,z_3]/(z_1^3-3z_1z_2-3z_3,z_2^3-3z_1z_2z_3+3z_3^2,z_3)\\
	\cong &\  {\F}\ [z_1,z_2]/(z_1^3-3z_1z_2,z_2^3).
	\end{align*}
	There is a natural map $\iota\colon A\rightarrow C$ defined by $\iota(z_i)=\overline{e_i}$ where $\overline{e_i}$ is the equivalence class of $e_i$ in $C$.  The kernel of $\pi$ is the ideal $(e_1,e_2,e_3)/(\hat{e}_1,\hat{e}_2,\hat{e}_3)=\iota({\mathfrak m_A})\cdot C\subset C$.  It is straightforward to check that $\dim_{\sf k}(C)=54=9\cdot 6=\dim_{\sf k}(A)\cdot\dim_{\sf k}(B)$. By Lemma \ref{coexactfreelem}, $C$ is a free extension with base $A$ and fiber $B$.
	Here $B$ has Hilbert function $H(B)=(1,2,2,1)$.  Since $B$ has codimension two, it must have a strong Lefschetz element. The algebra $C$ has Hilbert function $H(C)=(1,3,6,8,9,9,8,6,3,1)=H(A)\otimes H(B)$, and is known \footnote{See Proposition 4.26 in \cite{H-W}} to have the strong Lefschetz property.  The Hilbert function $H(A)=(1,1,2,1,2,1,1)$ is not unimodal, so $A$ cannot be strong Lefschetz.
 	Although $A$ is not strong Lefschetz, the non-homogeneous element $\ell=z_1+z_2\in\mathfrak{m}_A$ has strong Lefschetz Jordan type. To see this, note that the associated graded algebra $\Gr_{\mathfrak m_A}A$, viewed with the standard grading is
	$$\Gr_{\mathfrak m_A}(A)\cong {\F}[z_1,z_2]/(z_1^7,z_1z_2,z_2^3), \ \deg(z_i)=1$$
	with Hilbert function $H(\Gr_{\mathfrak m_A}(A))=(1,2,2,1,1,1,1)$ and conjugate partition $H(\Gr_{\mathfrak m_A}(A))=(7,2)=H(A)^\vee$. Since $Gr_{\mathfrak m_A}(A)$ has codimension two, Lemma~\ref{heighttwolem} implies it must have a strong Lefschetz element, which can be taken to be $\ell=z_1+z_2$.  This shows that $P_\ell=H(\Gr_{\mathfrak m_A}(A))^\vee=H(A)^\vee$ and hence that $\ell\in\mathfrak{m}_A$ has SLJT in $A$.
\end{example} 
\par
It would be nice if we could show that the base $A$ and fiber $B$ having SLJT elements implies that a free extension $C$ has a SLJT element.  Note that if this were true then, in the preceding Example \ref{ex:SNU}, we would have a new proof that the ring $C={\sf k}[x_1,x_2,x_3]/(\hat{e}_1,\hat{e}_2,\hat{e}_3)$ is strong Lefschetz, a fact that is not at all obvious: as by Proposition \ref{prop:UnimodalSLJTSL} that $C$ is standard-graded, has a SLJT element and that also $H(C)$ is unimodal imply $C$ is SL.

\section{Application to Invariant Theory.}\label{invthsec}
We study relative coinvariant rings, determined from the subgroups $K\subset W$ as the quotient $A=R^K_W$ of the invariant ring $R^K$ of $K$ by the coinvariant ideal $\mathfrak{h}_W= (R^W_+)$ of $W$.  In Section \ref{invthsec3.1}, after an example, we show in Theorem \ref{thm:ParLef} that if $K$ is a non-parabolic subgroup of $W$, then $R^K_W$ does not have the strong Lefschetz property. If $W$ is a Weyl group and $K\subseteq W$ is a parabolic subgroup, then $R^K_W$ has the strong Lefschetz property for geometric reasons, and we expect this to hold for all Coxeter groups.  On the other hand, we show in Example \ref{ex:InvEx4} that not every complex reflection group $W$ has the property that the relative coinvariants for every parabolic subgroup of $W$ is strong Lefschetz. In Proposition \ref{Gmmnprop}
we show that for $W=G(m,m,n)$ and $ K=G(m,m,n-1)$ the relative coinvariant ring is not strong Lefschetz, but has an element of strong Lefschetz Jordan type (SLJT).\par
In Section \ref{A(m,n)sec} we study the relative coinvariant rings $A(m,n)$ determined by  $K=\mathfrak S_n$, the symmetric group, as subgroup of $W=G(m,1,n)$, first giving examples of $A(3,3)$ and $A(3,4)$. Then in Remark \ref{Amnrem} we specify the Hilbert polynomial of $A(m,n)$ and note that it counts certain restricted partitions of integers $j$ (there are $m^n$ partitions in all), studied by G. Almkvist and others. We note that the study of the ring $A(m,n)$ is connected to plethysm (Remark \ref{Almkrem}).
\subsection{Lefschetz properties for coinvariant rings.}\label{invthsec3.1}
Recall that for a finite subgroup $W\subset\Gl(V)$ acting linearly on a vector space $V={\sf k}^n$, there is a corresponding action of $W$ on the polynomial ring $R=\Sym(V^*)$ according to the usual prescription $w\cdot f(v)=f(w^{-1}(v))$.  The polynomials invariant under this action form a subring $R^W\subset R$.  Recall that the Hilbert ideal $\mathfrak{h}_W\subset R$ is the ideal generated by invariant polynomials of strictly positive degree, and the quotient $R_W=R/\mathfrak{h}_W$ is what we call the \emph{coinvariant ring} of $W$. 

For any subgroup $K\subseteq W$, note there is an inclusion of invariants in the opposite direction $R^W\subset R^K$, and hence an inclusion of Hilbert ideals $\mathfrak{h}_W\subset\mathfrak{h}_K$.  Let $\pi\colon R_W\rightarrow R_K$ be the natural projection of coinvariant rings.  We define the \emph{relative coinvariant ring}\footnote{$R^K_W$ is not to be confused with the \emph{invariant coinvariant ring} defined as subring of $K$-invariant coinvariants $\left(R_W\right)^K$.  See \cite{Smith1} for more on the invariant coinvariant ring.}
$$R^K_W=\left(R^K\right)_W=\frac{R^K}{\mathfrak{h}_W\cap R^K}.$$
The natural inclusion $\hat{\iota}\colon R^K\hookrightarrow R$ passes to a map of quotient rings $\iota\colon R^K_W\rightarrow R_W$.  

The following result was proved in \cite[Theorem 2.20]{IMM2}.  See also \cite{Smith1}. 
\begin{proposition}
	\label{prop:FEInvariant}
With $K\subseteq W\subset\Gl(V)$, $V={\sf k}^n$, and $R=\Sym(V^*)$ as above:  if $R^K$ is polynomial, then $C=R_W$ is a free extension with base $A=R^K_W$ and fiber $B=R_K$.  Conversely, assuming additionally that $R^W$ is polynomial, then if $C=R_W$ is a free extension with base $A=R_W^K$ and fiber $B=R_K$, then $R^K$ must also be a polynomial ring. 
\end{proposition}
Recall that in the non-modular case, i.e. $|W|\in{\sf k}^*$, the invariant ring $R^W$ is polynomial if and only if $W$ is generated by reflections, i.e. $W$ is a finite reflection group.  In the modular case, i.e. where $\cha{\sf k}$ divides $|W|$, then $R^W$ polynomial only implies that $W$ is a reflection group--the other implication does not hold.

To prove that the coinvariant ring $R_W$ of a finite reflection group has the strong Lefschetz property, one could try the following strategy using free extensions and induction on the rank of $W$ (i.e. the dimension of the vector space on which $W$ effectively acts). For the base case, strong Lefschetz is known for coinvariant rings of rank one reflection groups--these rings have the form ${\sf k}[t]/(t^n)$.  For the induction step, find a reflection subgroup of smaller rank, say $K\subset W$, and show that its relative coinvariant ring $A=R^K_W$ (the base) has strong Lefschetz.  By the induction hypothesis the coinvariant ring $R_K$ (the fiber) has the strong Lefschetz property.  Conclude that the coinvariant ring $R_W$ has strong Lefschetz by Proposition \ref{thm:FELef}.  This strategy was used by the first author to prove that coinvariant rings of Coxeter groups (over $\R$) have strong Lefschetz \cite[Theorem 2]{McD}.\par
On the other hand, the following example shows that the relative coinvariant ring $R_W^K$ may not have the strong Lefschetz property for every choice of reflection subgroup $K\subset W$, even for Coxeter groups $W$.  We denote by ${\mathfrak S}_n$ the permutation group of $\{1,2,\ldots ,n\}$, and if $G'$ acts as automorphisms on $G$ we denote by $G\rtimes G'$ the semidirect product.

\begin{example}
	\label{ex:InvEx3}
	Let ${\F}=\R$ and let $W=\left\langle \left(\begin{array}{rr} -1 & 0\\ 0 & 1\\ \end{array}\right),\left(\begin{array}{rr} 1 & 0\\ 0 & -1\\ \end{array}\right)\right\rangle\rtimes\mathfrak{S}_2$. Then $W$ is the irreducible rank two Coxeter group of type $B$.  Take $K=\left\langle \left(\begin{array}{rr}-1 & 0\\ 0 & -1\\ \end{array}\right)\right\rangle\rtimes\mathfrak{S}_2\subset W$.  Note that $K$ is a (reducible) rank two Coxeter group, the product of two Coxeter groups ${\mathfrak S}_2$ of type $A_1$; and $K$ is non-parabolic (see definition below). The $W$-invariants are $R^W=\R[x^2+y^2,x^2y^2]$ and the $K$-invariants are $R^K=\R[x^2+y^2,xy]$, hence the relative coinvariants are
	\begin{equation*}
	A=R^K_W=  \frac{\R[x^2+y^2,xy]}{(x^2+y^2,x^2y^2)}
	\cong \frac{\R[x^2+y^2,t]}{(x^2+y^2,t^2)} \cong \frac{\R[t]}{(t^2)},  \, {\iota (t)= xy},
	\end{equation*}
	of Hilbert function $H(A)=(1,0,1)$.
	Since $\deg(t)=2$ the ring $A=R^K_W$ has no linear elements, so is not strong Lefschetz, although the coinvariant rings $R_W, R_K$, having codimension two, are each SL.  However $t=xy\in A_2$ is a SLJT element, since $P_t=(2)=H(A)^\vee$. 
	In the Shephard-Todd classification \cite{ShepTodd} $W=G(2,1,2)$ and $K=G(2,2,2).$\footnote{For a generalization to $W=G(m,p,n), K=G(m,p',n), p|p'|m$ see \cite[Example 2.16]{IMM3}.}
\end{example}

Given a reflection group $W\subset\Gl(V)$, a subgroup $K\subset W$ is called \emph{parabolic} if it is the stabilizer subgroup of some subspace $H\subset V$,  so $K=\Stab (H)=\left\{w\in W\left|w(h)=h, \ \forall \ h\in H\right.\right\}$.  Parabolic subgroups are always generated by reflections (R. Steinberg \cite[Theorem~1.5]{St}, see also G.I.~Lehrer and D.E.~Taylor \cite[\S 9.7]{LT}). The following result, though not difficult to prove, seems to be new and is the main result of this section.
\begin{theorem}
	\label{thm:ParLef}
	Assume that $K\subset W\subset\Gl(V)$ are non-modular finite reflection groups (non-modular so their invariant rings are polynomial).  If $K$ is not parabolic, then the relative coinvariant ring $R^K_W$ cannot have the strong Lefschetz property.
\end{theorem}
\begin{proof}
	Assume that $K\subset W$ is not parabolic.  Let $H\subset V$ be the largest subspace fixed by $K$.  If $\dim_{\F}H=0$, then $R^K_W$ contains no linear forms, hence $R^K_W$ cannot have the strong Lefschetz property.  Assume $\dim_{\F}H>0$ and set $K_1=\Stab_W(H)$, the stabilizer of $H$ in $W$.  Note that $K\subset K_1$, but $K\neq K_1$ since $K$ is not parabolic.  Since $K\subset K_1$ we have an inclusion of invariant (polynomial) rings $\tau\colon R^{K_1}\rightarrow R^K$.  Let $I\subset R^{K_1}$ be the ideal $I=(R^W)_+\cdot R^{K_1}$, and set $A=R^{K_1}/I=R^{K_1}_W$, $C=R^K/\left(\tau(I)\right)=R^K_W$, and $B=R^{K}/(\tau(R^{K_1}_+))=R^{K}_{K_1}$.  Then we have maps $\bar{\tau}\colon A\rightarrow C$ and $\pi\colon C\rightarrow B$.  The Hilbert polynomial of $R^G_H$ is the quotient of Hilbert polynomials:
	$${\sf p}(R^G_H)(t)=\frac{{\sf p}(R_H,t)}{{\sf p}(R_G,t)}.$$ 
It follows from Lemma \ref{coexactfreelem} 
	 that $C=R^K_W$ is a free extension over $A=R^{K_1}_W$ with nonzero fiber $B=R^{K}_{K_1}$  (since $K\neq K_1$) satisfying $B_1=0$, so $B$ is not SL.
	
	Let $\ell\in C_1$ be any linear form, and note that $\ell=\bar{\tau}(\ell_A)$ for some $\ell_A\in A_1$.  Then since $C$ is an $A$ module, the map $\times \ell^{a+b}\colon C_0\rightarrow C_{a+b}$ must have rank zero as $a+b>a$. This implies that $C$ cannot have the strong Lefschetz property.  
\end{proof}\par
\begin{remark}\label{G/Prem}
Over $\R$, the reflection groups $W\subset\Gl(V)$ that preserve a lattice $L\subset V$ are called \emph{Weyl groups}, and these groups are associated to certain smooth complex projective algebraic varieties the flag varieties $G/B$.  Furthermore, every parabolic subgroup $K\subset W$ of a Weyl group is associated to another smooth projective algebraic variety called a partial flag variety $G/P$ (e.g. a Grassmannian variety).  A classical result of Borel \cite{Borel} is that the coinvariant ring of a Weyl group $R_W$ over the ground field $\F=\QQ$ is isomorphic to the cohomology ring of its associated flag variety $H^*(G/B,\QQ)$ with coefficients in $\QQ$ (see also \cite[Proposition 1.3]{BGG}).  It was further shown by Bernstein-Gelfand-Gelfand \cite[Theorem 5.5]{BGG} that the cohomology ring of the partial flag variety $H^*(G/P,\QQ)$ with coefficients in $\QQ$ can then be identified with the relative coinvariant ring $R_W^K$ over $\F=\QQ$.  On the other hand, the hard Lefschetz theorem in algebraic geometry implies that the cohomology ring $H^*(X,\QQ)$ of any smooth complex projective algebraic variety $X$ has the strong Lefschetz property. This implies that for any Weyl group $W$ and parabolic subgroup $K\subset W$, the relative coinvariant ring $R^K_W$ has the strong Lefschetz property over $\F=\QQ$, and hence over any field of characteristic zero.  This fact was also pointed out in \cite[Theorem 2]{MNW}. 
\end{remark}
\begin{remark*}[Sources]
	For more on real reflection groups, also known as Coxeter groups, we refer the reader to J.E. Humphrey's book \cite{Humphreys}.  A thorough discussion of the connection between Weyl groups and flag varieties is in T. Springer's book \cite{Springer}.  For a discussion of the hard Lefschetz theorem for K\"ahler manifolds (of which smooth complex projective varieties are a subclass), see Huybrechts' book \cite{Huybrechts} or Griffiths and Harris' \cite{GrifHarris}.   The identification of the coinvariant ring and the relative coinvariant ring with the cohomology rings of the full and partial flag varieties is given in the paper \cite{BGG} of I.N. Bernstein, I.M. Gelfand, and S.I.~Gelfand.  Finally there is a subtlety we wish to point out:  the definition of parabolic subgroups that we gave above is the one traditionally given for complex reflection groups, e.g. \cite[Definition 9.1]{LT}.  On the other hand, for real reflection groups, a parabolic subgroup is usually defined to be a subgroup that is generated by some subset of a fixed set of minimal generating reflections.  That the two notions are equivalent, is shown in a paper by D.E. Taylor \cite[Theorem 4.2]{Taylor}. 
\end{remark*}

The following question is natural, and appears to be open, even for real reflection groups (i.e. Coxeter groups):
\begin{question}
	\label{ques:4}
	Given a non-modular finite reflection group $W\subset\Gl(V)$, for which parabolic subgroups $K\subseteq W$ does the relative coinvariant ring $R_W^K$ have the strong Lefschetz property?  
\end{question} 
For real reflection groups, we conjecture that the answer to Question \ref{ques:4} is ``all of them''.
\begin{conjecture}
	\label{conj:CoxPar}
	If $W\subset\Gl(V)$ is a real reflection group, then for every parabolic subgroup $K\subseteq W$, its relative coinvariant ring $R^K_W$ has a strong Lefschetz element.
\end{conjecture}
The next example shows that this is not true for complex reflection groups.

\begin{example}\cite[Example 1.2]{IMM1}
	\label{ex:InvEx4}
	Let ${\F}=\C$, and let 
	$$W=\left\{\left.\left(\begin{array}{ccc} \lambda_1 & 0 & 0\\ 0 & \lambda_2 & 0\\ 0 & 0 & \lambda_3\\ \end{array}\right)\right|\lambda_i^3=1,\lambda_1\lambda_2\lambda_3=1\right\}\rtimes\mathfrak{S}_3.$$
 Here $W$ is the complex reflection group called $G(3,3,3)$ in the Shephard-Todd classification \cite{ShepTodd}.  Let $K=G(3,3,2)\subset W$ be the parabolic subgroup that fixes the last coordinate, so $K$ is the semidirect product
	$$K=\left\{\left.\left(\begin{array}{ccc} \lambda & 0 & 0\\ 0 & \lambda^{-1} & 0\\ 0 & 0 & 1\\ \end{array}\right)\right|\lambda^3=1\right\}\rtimes\mathfrak{S}_2.$$
Each of $W$ and $K$ acts on $R=\C[x,y,z]$ in the obvious way, and their invariants are given by $R^W=\C[x^3+y^3+z^3,x^3y^3+x^3z^3+y^3z^3,xyz]$ and $R^K=\C[x^3+y^3,xy,z]=\C[a,b,c], a=x^3+y^3, b=xy,c=z$.  The relative coinvariant ring is 
\begin{align*}
A=R^K_W &=  \frac{\C[x^3+y^3,xy,z]}{(x^3+y^3+z^3,x^3y^3+x^3z^3+y^3z^3,xyz)}\\
 &=\frac{\C\{a,b,c\}}{(a+c^3,b^3+ac^3,bc)}\\
 &\cong \frac{\C\{b,c\}}{(b^3-c^6,bc)}
\end{align*}
where the variables are weighted by $w(b,c)=(2,1)$.  From this we can see the Hilbert function $H(A)=(1,1,2,1,2,1,1)$ which is non-unimodal, hence $A$ cannot have an SL element.  On the other hand it follows from Lemmas \ref{heighttwolem} and \ref{lem:heighttwoHF} that $A$ does have an element of SLJT.  Explicitly, one can check that the non-homogeneous element $\ell=b+c$ has Jordan strings
\begin{align*}
S_1= & \left\{1,\ell,\ell^2,\ell^3,\ell^4,\ell^5,\ell^6\right\}\\
S_2= & \left\{\alpha,\ell\alpha\right\} & \alpha=b-c^4,
\end{align*}
thus $P_\ell=(7,2)=H(A)^\vee$, so the element $ \ell$ has SLJT. 



\end{example}
While this example shows that not all relative coinvariant rings are SL, perhaps it also suggests that they could all have elements of  SLJT:

\begin{conjecture}
	\label{conj:3} 
	Suppose $K\subset W\subset \Gl(V)$, are non-modular finite reflection groups.  Then the relative coinvariant ring $A=R^K_W$ has an element of strong Lefschetz Jordan type.
\end{conjecture}

The following result generalizes Example \ref{ex:InvEx4} and supports Conjecture \ref{conj:3}. Recall that for ${\sf k}=\C$ there is a three-parameter family of complex reflection groups $G(m,p,n)$ for integers $m$, $p$, and $n$ where $p|m$, defined by 
$$G(m,p,n)=\left\{\left.\left(\begin{array}{cccc} \lambda_1 & 0 & \cdots & 0\\ 
0 & \lambda_2 & \cdots & 0\\
\vdots & \vdots & \ddots & \vdots\\
0 & 0 & \cdots & \lambda_n\\ \end{array}\right)\right| \lambda_i^m=1 \ \forall \ i, \   \left(\lambda_1\cdot\lambda_2\cdots\lambda_n\right)^{\frac{m}{p}}=1\right\}\rtimes\mathfrak{S}_n.$$
In words, $G(m,p,n)$ is the group of $n\times n$ permutation matrices in which the non-zero entries are $m^{th}$ roots of unity, the product of which is an $\left(\frac{m}{p}\right)^{th}$ root of unity. 
It satisfies $G(m,p,n)=T\rtimes {\mathfrak S}_n$, the semidirect product of $T$ and the permutation group  ${\mathfrak S}_n$ \cite[Remark 7.13]{Ne},\cite{ShepTodd}.  
\begin{proposition}\label{Gmmnprop}\par
	Let ${\sf k}=\C$, $W=G(m,m,n)$ and consider the parabolic subgroup $K=G(m,m,n-1) \subset G(m,m,n)=W$ acting on $R={\sf k}[x_1,\ldots,x_n]$ in the usual way.  Then the relative coinvariant ring $R^K_W$ is not SL, but it has an element of SLJT.
\end{proposition}
	
\begin{proof}
Let $E_i=E_i(x_1,\ldots,x_n)=\sum_{1\leq k_1<\cdots<k_i\leq n}x_{k_1}\cdots x_{k_i}$ be the $i^{th}$ elementary symmetric polynomial in variables $x_1,\ldots,x_n$, and denote by $\hat{E}_i=E_i(x_1^m,\ldots,x_n^m)$ the $i^{th}$ elementary symmetric polynomial in variables $x_1^m,\ldots,x_n^m$.  Then the invariant ring for $W$ is the polynomial ring 
$$R^W={\sf k}[\hat{E}_1,\ldots,\hat{E}_{n-1},E_{n}].$$
Letting $e_i=e_i(x_1,\ldots,x_{n-1})$ denote the $i^{th}$ elementary symmetric polynomial in one less variable, and $\hat{e}_i=e_i(x_1^m,\ldots,x_{n-1}^m)$ we get the invariant ring for $K$: 
$$R^K={\sf k}[\hat{e}_1,\ldots,\hat{e}_{n-2},e_{n-1},x_n].$$
Note we have the relations
\begin{align*}
\hat{E}_i= & \hat{e}_i+x_n^m\cdot\hat{e}_{i-1}, \ 1\leq i\leq n-1\\
E_{n}= & x_n^m\cdot\hat{e}_{n-1}
\end{align*}
The corresponding relative coinvariant ring is
\begin{align*}
A=R^K_W= & \frac{{\sf k}[\hat{e}_1,\ldots,\hat{e}_{n-2},e_{n-1},x_n]}{\left(\hat{E}_1,\ldots,\hat{E}_{n-1},E_n\right)}=\frac{{\sf k}[\hat{e}_1,\ldots,\hat{e}_{n-2},e_{n-1},x_n]}{\left(\hat{e}_1+x_n^m,\ldots,\hat{e}_{n-1}+x_n^m\hat{e}_{n-2},x_n\cdot e_{n-1}\right)} \\
\cong & \frac{{\sf k}[e_{n-1},x_n]}{\left(\hat{e}_{n-1}-\left(x_n^m\right)^{n-1},x_n\cdot e_{n-1}\right)}= \frac{{\sf k}[a,b]}{\left(a^m-b^{m\cdot({n-1})},b\cdot a\right)}
\end{align*}
where the variables $a=e_{n-1}$ and $b=x_n$ have weights ${\sf w}(a,b)=(n-1,1)$.  Its Hilbert polynomial is 
\begin{equation}\label{shorteq}
{\sf p}(A,t)=\sum_{i=0}^{m(n-1)}t^i+\sum_{i=1}^{m-1}t^{i(n-1)},
\end{equation}
 and which is non-unimodal, hence $A$ is not SL.  On the other hand it follows from Lemmas  \ref{heighttwolem} and \ref{lem:heighttwoHF} that $A$ must have an element of SLJT.
 
\end{proof}\par
Thus, all of these Hilbert functions $H(A)$ are non-unimodal (when both $m,n\ge 3$) and have the form $H(A)=(1, \ldots, 1, 2, 1, \ldots ,1,2,1, \ldots, 1)$. 
\vskip 0.2cm\vskip 0.2cm\noindent
{\bf Examples}: (Example \ref{ex:InvEx4}): $H(R_{G(3,3,3)}^{G(3,3,2)})=(1,1,2,1,2,1,1)$
\par
(Increasing $m$): $H(R_{G(4,4,3)}^{G(4,4,2)})=(1,1,2,1,2,1,2,1,1)$\par
$H(R_{G(5,5,3)}^{G(5,5,2)})=(1,1,2,1,2,1,2,1,2,1,1)$;\vskip 0.2cm
(Increasing n): $H(R_{G(3,3,4)}^{G(3,3,3)})=(1,1,1,2,1,1,2,1,1,1)$\par
$H(R_{G(3,3,5)}^{G(3,3,4)})=(1,1,1,1,2,1,1,1,2,1,1,1,1)$;\par
(Both increasing): $H(R_{G(4,4,5)}^{G(4,4,4)})=(1,1,1,1,2,1,1,1,2,1,1,1,2,1,1,1,1)$.\vskip 0.2cm
\par
\begin{remark}[Symmetric decomposition]  When the algebra of relative coinvariants $A=R^K_W$ is Gorenstein, the localization $\mathscr{A}=\kappa(A)$ of page \pageref{kappapage} is Gorenstein, and the theory of symmetric decompositions of the associated graded algebras of Gorenstein algebras \cite{I1} can be applied. Then the symmetric components -- successive quotients of a certain filtration -- of $\Gr_{\mathfrak m_{\mathscr{A}}}(\mathscr{A})$ correspond to quotients of a filtration of $A$ that are invariant under the action of $K$, and there is a symmetric Hilbert function decomposition of $H(\mathscr{A})$, and an induced decomposition of $H(A)$.  We leave developing this for a subsequent work.
\end{remark}

\subsection{Relative coinvariant rings $A(m,n)$ and SLJT.}\label{A(m,n)sec}
We give further infinite families of relative coinvariant rings that (sometimes) have non-homogeneous elements of strong Lefschetz Jordan type. The ring $A(m,n)$ is the relative coinvariant ring $R^K_W$ with $K=\mathfrak S_n$, the symmetric group, and $W=G(m,1,n)$. We begin with $A(3,3)$.
\begin{example}
	\label{ex:G313S3}
	Let ${\F}=\C$ and let $W$ be the complex reflection group $G(3,1,3)$.  Let $K=\mathfrak{S}_3\subset W$, let $e_i=e_i(x_1,x_2,x_3)$ be the $i$-th elementary symmetric polynomial in the variables $x_1$, $x_2$, $x_3$, and denote by $\hat{e}_i=e_i(x_1^3,x_2^3,x_3^3)$ be the $i$-th elementary symmetric polynomial in $x_1^3$, $x_2^3$, $x_3^3$.  The coinvariant rings are given by
	\begin{equation*}
	R_W=  \frac{\C[x,y,z]}{(x^3+y^3+z^3,x^3y^3+x^3z^3+y^3z^3,x^3y^3z^3)}
\cong  \frac{\C[x_1,x_2,x_3]}{(\hat{e}_1,\hat{e}_2,\hat{e}_3)}
	\end{equation*}
	and 
	\begin{equation*}
	R_K=  \frac{\C[x,y,z]}{(x+y+z,xy+xz+yz,xyz)}
	\cong  \frac{\C[x_1,x_2,x_3]}{(e_1,e_2,e_3)}.
	\end{equation*}
	The relative coinvariants are, letting $ a=e_1,b=e_2, c=e_3$,
	\begin{equation*}
	A=A(3,3)=R^K_W=  \frac{\C[e_1,e_2,e_3]}{(\hat{e}_1,\hat{e}_2,\hat{e}_3)}
	=  \frac{\C[a,b,c]}{(a^3-3ab+3c,b^3-3abc+3c^2,c^3)}.
	\end{equation*}
	(The last equality was computed using the relations in Example \ref{ex:SNU}).
	By Proposition~\ref{prop:FEInvariant}, $R_W$ is a free extension with base $A=R_W^K$ and fiber $B=R_K$, and by Lemma~\ref{coexactfreelem} we have a vector space isomorphism $R_W\cong_{\sf k} R_W^K\otimes_{\sf k} R_K$.  The Hilbert polynomial of $A$ satisfies
	\begin{equation*}
	{\sf p}(R^K_W,t)=  \frac{{\sf p}(R_W,t)}{{\sf p}(R_K,t)}
	=  \frac{\frac{(1-t^3)(1-t^6)(1-t^9)}{(1-t)^3}}{\frac{(1-t)(1-t^2)(1-t^3)}{(1-t)^3}}
	=  (1+t+t^2)(1+t^2+t^4)(1+t^3+t^6).
	\end{equation*} 
	The Hilbert function is
	$H(R^K_W)=(1,1,2,2,3,3,3,3,3,2,2,1,1)$. Since $A=R^K_W$ has embedding dimension two
	(and its associated graded has Hilbert function $H(\Gr_{\mathfrak m_A}(A))=(1,2,3,3,3,3,3,2,2,2,1,1,1)$) it has a non-homogeneous element of SLJT by Lemma~\ref{heighttwolem}. However, a straightforward calculation shows that 
	$\ell= a$ is a (homogeneous) element of $A$ with strings having cyclic generators 
	$\{1, b, b^2\}$, and that $P_\ell=(13,9,5)=H(A)^\vee$, so $a$ is a strong Lefschetz element for $A$.
\end{example}

\begin{example}[Ring $A(3,4)$]
	\label{ex:G314S4}
	Similarly to Example \ref{ex:G313S3}, we take $W=G(3,1,4)$, 
	and let $K=\mathfrak{S}_4$.  Consider $e_i=e_{i,4}$ and $\hat{e}_i=\hat{e}_{i,4}$ but with degree $m=4$.  Then we have coinvariant rings
	$$R_W=\frac{\C[x_1,\ldots,x_4]}{(\hat{e}_1,\ldots,\hat{e}_4)} \text { and  } R_K=\frac{\C[x_1,\ldots,x_4]}{(e_1,\ldots,e_4)}.$$
	The relative coinvariant ring $A$ is given by 
	\begin{align*}
	A(3,4)=R_W^K= & \frac{\C[e_1,\ldots,e_4]}{(\hat{e}_1,\ldots,\hat{e}_4)}\\
	= & \frac{\C[a,b,c,d]}{(a^3-3ab+3c,b^3-3abc+3a^2d-3bd+3c^2,c^3-3bcd+3ad^2,d^3)}.
	\end{align*}
	As before, the Hilbert polynomial can be computed as
	\begin{equation*}
	{\sf p}(R^K_W,t)=  \frac{\frac{(1-t^3)(1-t^6)(1-t^9)(1-t^{12})}{(1-t)^4}}{\frac{(1-t)(1-t^2)(1-t^3)(1-t^4)}{(1-t)^4}}
	=  (1+t+t^2)(1+t^2+t^4)(1+t^3+t^6)(1+t^4+t^8)
	\end{equation*}
	Its Hilbert function is 
	$$H(R^K_W)=(1,1,2,2,4,4,5,5,7,6,7,6,7,5,5,4,4,2,2,1,1).$$
	Both $C=R_W$ and $B=R_K$ are strong Lefschetz by \cite[Proposition 4.6]{H-W}.
	We see from the non-unimodality of $H(R^K_W)$ that $A=R_W^K$ is not strong Lefschetz. The embedding dimension of the local ring $\mathscr{A}=\kappa(A)$ is three. Does $A$ have an element of SLJT? \par
\end{example}
\begin{remark}\label{Amnrem}
We generalize Examples \ref{ex:G313S3} and \ref{ex:G314S4}:  Fix positive integers $m,n$, and take $W=G(m,1,n)$ and $K=\mathfrak{S}_n\subset W$.  Let $e_i=e_{i,n}$ be the $i$-th elementary symmetric function in variables $x_1,\ldots,x_n$ and let $\hat{e}_i=\hat{e}_{i,n}$ be the $i$-th elementary symmetric function in $ (x_1^m,\ldots,x_n^m)$. Then the relative coinvariant ring is given by 
\begin{equation}\label{Amneq}
A(m,n)\coloneqq R^{\mathfrak{S}_n}_{G(m,1,n)}=\frac{\C[e_1,\ldots,e_n]}{(\hat{e}_1,\ldots,\hat{e}_n)}.
\end{equation}
Its length $|A(m,n)|=m^n$ and its Hilbert polynomial is, letting $N=(m-1){\binom{n+1}{2}}$,
\begin{equation}\label{HilbAmneq}
{\sf p}(A(m,n),t)=  \frac{\prod_{i=1}^n(1-t^{im})}{\prod_{i=1}^n(1-t^i)}
=  \prod_{i=1}^n\left(1+t^i+\cdots+t^{i(m-1)}\right)=\sum_{j=1}^N c(j,m,n)t^j
\end{equation}\par\noindent 
 This polynomial has a combinatorial interpretation:  it is the generating function for the number of partitions $c(m,n,j)$ of $j$ into at most $n$ parts, and whose number of parts of any given size does not exceed $m-1$.  Such restricted partition functions have been studied extensively in combinatorics, e.g. \cite{Alm0,Alm1,Alm2,Stanley1}.  In particular,  G.~Almkvist conjectured in 1985 \cite{Alm0}:
\begin{conjecture}
	\label{conj:Almkvist}
	For fixed $m$, the polynomial ${\sf p}(A(m,n),t)$ has unimodal coefficients for all $n$ sufficiently large.
\end{conjecture}
This conjecture has been verified by G. Almkvist \cite{Alm2} for values of $m$, $3\leq m\leq 20$, and also $m=100$ and $m=101$.  He notes that it had been shown
for $m=2$ and all $n$ in several ways \cite{Hu,OdlRi}; by the Note at the end of Remark \ref{Almkrem}) $A(2,n)$ satisfies strong Lefschetz for all $n$. So we extend Almkvist's conjecture in the obvious way to our algebraic context:
\begin{conjecture}
	\label{conj:A2.0}
	\begin{enumerate}
		\item For fixed $m$, the graded Artinian complete intersection $A(m,n)$ has the strong Lefschetz property for all $n$ sufficiently large and\par
	  \item $A(m,n) $ has maximum Jordan type consistent with its Hilbert function. In particular, $A(m,n)$ is strong Lefschetz if and only if $H(A(m,n))$ is unimodal, and $A(m,n)$ is always weak Lefschetz.
	\end{enumerate}
\end{conjecture}
As before, the coinvariant ring $R_W=R_{G(m,1,n)}$ is a free extension of $A(m,n)$ with fiber $R_K=R_{\mathfrak S_n}$.
\par\noindent       
\end{remark}
\begin{remark}\label{Almkrem}[Using plethysm to calculate $A(2,n)$]
Let $R={\sf k}[x_1,\ldots,x_n]$, and consider $A(2,n)=R^{S_n}_{G(2,1,n)}$. A. Odlyzko and B. Richmond have shown that the Hilbert polynomial ${\sf p}(A(2,n),t)$ has unimodal coefficients for each $n$ \cite[Thm. 4 Cor]{OdlRi}.  Moreover viewing $W=G(2,1,n)$ as a Weyl group of type $B_n$ and $K=\mathfrak{S}_n\subset W$ as a parabolic subgroup for every $n$, we can identify $A(2,n)$ as the cohomolgy ring of a smooth complex projective algebraic variety as in Remark \ref{G/Prem}.  Thus, for $\cha{\sf k}=0$, we deduce that $A(2,n)$ is SL.  We use plethysm to compute a presentation of $A(2,n)$ as follows:
 
As before, let $e_i=e_i(x_1,\ldots,x_n)$ be the $i^{th}$ elementary symmetric polynomial in variables $x_1,\ldots,x_n$, and let 
$\hat{e}_i=e_i(x_1^2,\ldots,x_n^2)$ be the $i^{th}$ elementary symmetric polynomial in $x_1^2,\ldots,x_n^2$.  The symmetric function $\hat{e}_i$ is an example of a plethysm, sometimes written $p_2[e_i]$, where $p_2=p_2(x_1,\ldots,x_n)$ is the symmetric power function $x_1^2+\cdots+x_n^2$; see \cite{Kerb}. 
This plethysm can be shown to satisfy \par
\begin{equation}
p_2[e_i]=\hat{e}_i=\sum_{m=0}^{2i}(-1)^me_me_{2i-m}
\end{equation}
where $e_0=1$ and $e_{k}=0$ for $k>n$.  Thus, we have the presentation 
\begin{align*}
A(2,n)=R^{\mathfrak{S}_n}_{G(2,1,n)}= & {\sf k}[e_1,\dots,e_n]/(p_2[e_1],\dots,p_2[e_n])\\
= & {\sf k}[e_1,\dots,e_n]\left/\left(\Bigg\{ \left.\sum_{m=0}^{2i}(-1)^me_me_{2i-m}\right| i=1,\ldots,n\Bigg\} \right)\right.
\end{align*}
In particular, the embedding dimension of $A(2,n)$ is $\big\lfloor\frac{n}{2}\big\rfloor$.
For example, when $n=4$, we have 
\begin{align*}\label{A24eqn}
A(2,4)&={\sf k}[e_1,e_2,e_3,e_4]/(e_1^2-2e_2,e_2^2-2(e_1e_3-e_4),e_3^2-2e_2e_4,e_4^2)\\
&\cong {\sf k}[e_1,e_3,e_4]/\left(e_1^4-8(e_1e_3-e_4),e_3^2-e_1^2e_4,e_4^2\right).
\end{align*}
Here the variables have weights ${\sf w}(e_1,e_2,e_3,e_4)=(1,2,3,4)$, and $A(2,4)$ is a complete intersection of generator degrees  $D=(2,4,6,8)$ whose Hilbert function, from Equation \eqref{HilbAmneq}, is $H(A(2,4))=(1,1,1,2,2,2,2,2,1,1,1)$. Evidently, the local ring $\kappa(A(2,4))$
has embedding dimension two (and Hilbert function $H(\kappa(A))=(1,2,2,2,2,2,1,1,1,1,1)$), hence by Lemma \ref{heighttwolem}, $A(2,4)$ has an element (possibly non-homogeneous) of SLJT.  As we noted above, for geometric reasons it has a strong Lefschetz element. 
\end{remark}

\begin{example}\label{newremark}
Almkvist's conjecture concerns the Hilbert function of the relative coinvariant ring $A(m,n)=R^{S_n}_{G(m,1,n)}=R^{G(1,1,n)}_{G(m,1,n)}$. We consider, more generally, $A(m,p,n)=R^{G(p,p,n)}_{G(m,p,n)}$ where p divides m and $p \neq m$.

First, consider  $A=A(6,2,3)$,
$$A={\sf k}[x^2+y^2+z^2,x^2y^2+x^2z^2+y^2z^2,xyz]/(x^6+y^6+z^6,x^6y^6+x^6z^6+y^6z^6,x^3y^3z^3).$$

Its Hilbert function is $H(A)=(1,0,1,1,2,1,2,2,3,1,3,2,2,1,2,1,1,0,1)$, and its Hilbert polynomial is 
$${\sf p}(A,t)=\frac{(1-t^6)(1-t^{12})(1-t^9)}{(1-t^2)(1-t^4)(1-t^3)}=(1+t^2+t^4)(1+t^4+t^8)(1+t^3+t^6).$$

Observe that there is no $t^1$ term. Moreover, increasing n doesn't affect this. It is not hard to show that $A(6,2,n)=R^{G(2,2,n)}_{G(6,2,n)}$ has a non-unimodal Hilbert function for any $n\ge 2$. 

When $p=1$, $A(m,1,n)$ is exactly the family $A(m,n)$ considered before, which G. Almkvist conjectured to have unimodal Hilbert functions for fixed $m$ and $n$ large enough. For $p \geq 2$, by the same argument as above, there is no $t^1$ term when expanding out their Hilbert polynomial ${\sf p}(A,t)$.
The coinvariant ring
$$R_{G(m,p,n)}={\sf k}[x_1,\ldots,x_n]/\left(e_1(x_1^m,\ldots ,x_n^m),\ldots ,e_{n-1}(x_1^m,\ldots ,x_n^m), (x_1\cdots x_n)^{m/p}\right),$$ is a complete intersection of generator degrees $(m,2m,\ldots ,(n-1)m,nm/p)$, and $R_G(p,p,n)$ is a CI of generator degrees $(p,2p,\ldots,(n-1)p,n)$. We have, letting $k=m/p$, the length
$|A(m,p,n)|=k^n$, and its Hilbert polynomial is
\begin{align}
{\sf p}(A(m,p,n),t)&= \frac{\left(\prod_{i=1}^{n-1}(1-t^{im})\right)(1-t^{mn/p})}{\left(\prod_{i=1}^{n-1}(1-t^{ip})\right)(1-t^n)}\\
&=  \left(\prod_{i=1}^{n-1}\left(1+t^{ip}+\cdots+t^{(k-1)ip}\right)\right)\cdot (1+t^n+t^{2n}+\cdots + t^{(k-1)n})\notag\\
&=\sum_{j=1}^N c(j,m,p,n)t^j,
\end{align}
where $N=(k-1)\left(p\cdot \binom{n}{2}+n\right)$. When $p\ge 2$ there are no terms $t^1$ and $H(A(m,p,n))$ is non-unimodal, as $A(m,p,n)_1=0$. That $A(m,p,n)$ is not SL follows also from Theorem \ref{thm:ParLef}, as $G(p,p,n)$ is not a parabolic subgroup of $G(m,p,n)$ when $p\ge 2$.
\vskip 0.2cm \noindent
Here ${\sf p}(A(m,p,n),t)$ is the generating function for the number of partitions $c(m,p,n,j)$ of $j$ into at most $n$ parts, where each part is either a multiple $ip$ (for $1\le i<n$) or is $n$, and whose number of parts of any given size does not exceed $k-1$. \vskip 0.2cm\noindent 
{\bf Question}. Do these algebras have (non-homogeneous) elements of SLJT? 
\end{example}
{\begin{question}
Can we classify those non-standard graded Artinian algebras whose associated graded algebras have the same partition type as the algebra itself: so 
$$H(A)^\vee =H(\Gr_{\mathfrak{m}_A}(A))^\vee ?$$
This condition, which holds for ${\sf k}[x]/(x^3)$ with ${\sf w}(x)=2$, just requires that the two Hilbert functions
differ by only zeroes, and, possibly, a reordering. It does not hold for the CI $A={\sf k}[x,y]/(xy, x^2+y^3)$ with grading ${\sf w}(x,y)=(3,2)$ as $A=\langle 1,y,x,y^2,y^3\rangle$ so $H(A)=(1,0,1,1,1,0,1)$ but $\Gr_{\mathfrak{m}_A}(A))$ has Hilbert function $(1,2,1,1)$, nor does it hold for the CI algebra of Example \ref{ex:leq} (see Lemma \ref{lem:heighttwoHF}).
Would it hold for any relative coinvariant ring $A=R^K_W$ for $K\subset W\subset\Gl(V)$ complex reflection groups?
Note that we need this in Examples~\ref{ex:SNU}, \ref{ex:G313S3}, and in Remark \ref{Almkrem}.
\end{question}
\begin{ack} This paper began with discussions after a talk Chris McDaniel gave at Northeastern University in Spring 2017 and his presentation of Example~\ref{ex:InvEx4}. Shujian Chen in a Junior-Senior thesis project supervised by the third author found with Chris further examples of relative covariants having non-unimodular Hilbert functions; Pedro Marques joined.  The authors would like to thank Junzo Watanabe, Uwe Nagel, and Alexandra Secealanu for helpful comments/examples, Donald King and Gordana Todorov for their comments, Emre Sen and Ivan Martino for discussions, and Richard Stanley and Stephanie van Willigenburg for their helpful responses to questions.\par
Pedro Macias Marques was partially
supported during the work by CIMA -- Centro de Investiga\c{c}\~{a}o em Matem\'{a}tica e
Aplica\c{c}\~{o}es, Universidade de \'{E}vora, project
PEst-OE/MAT/UI0117/2014 (Funda\c{c}\~{a}o para a Ci\^{e}ncia e
Tecnologia).
\end{ack}
\small
\addcontentsline{toc}{section}{References}
\bibliographystyle{amsplain}
\bibliography{CIMM-May8ref}

\end{document}